\documentclass[12pt]{amsart}
\usepackage{amssymb,latexsym,comment,url}
\usepackage[all]{xy}

\newfont{\wncyr}{wncyr10 at 12pt}
\newfont{\wncyrten}{wncyr10 at 10pt}

\newcommand{\Br}{\operatorname{Br}}
\newcommand{\Gal}{{\operatorname{Gal}}}
\newcommand{\GL}{{\operatorname{GL}}}

\newcommand{\Hom}{{\operatorname{Hom}}}
\newcommand{\CT}{{\operatorname{CT}}}

\newcommand{\Ind}{{\operatorname{Ind}}}
\newcommand{\End}{{\operatorname{End}}}
\newcommand{\SL}{{\operatorname{SL}}}
\newcommand{\Mat}{{\operatorname{Mat}}}
\newcommand{\im}{\operatorname{Im}}
\newcommand{\inv}{\operatorname{inv}}
\newcommand{\cyc}{{\operatorname{cyc}}}
\newcommand{\loc}{\operatorname{loc}}
\newcommand{\rank}{\operatorname{rank}}
\newcommand{\F}{{\mathbb F}}
\newcommand{\phihat}{{\widehat{\phi}}}
\newcommand{\psihat}{{\widehat{\psi}}}
\newcommand{\Kbar}{{\overline K}}
\newcommand{\PP}{{\mathbb P}}
\newcommand{\R}{{\mathbb R}}
 
\newcommand{\cor}{{\textup{Cor}}}
\newcommand{\ra}{\longrightarrow}
\newcommand{\res}{{\textup{Res}}}
\newcommand{\Q}{{\mathbb Q}}
\newcommand{\Z}{{\mathbb Z}}
\newcommand{\OK}{{{\mathcal O}_K}}
\newcommand{\pp}{{\mathfrak p}}
\newcommand{\bb}{{\mathfrak b}}
\newcommand{\isom}{\cong}
\newcommand{\Sha}{\mbox{\wncyr Sh}}
\newcommand\numberthis{\addtocounter{equation}{1}\tag{\theequation}}

\newtheorem{Proposition}{Proposition}[section]
\newtheorem{Theorem}[Proposition]{Theorem}
\newtheorem{Lemma}[Proposition]{Lemma}

\theoremstyle{definition}
\newtheorem{Definition}[Proposition]{Definition}
\newtheorem{Remark}[Proposition]{Remark}
\newtheorem{Example}[Proposition]{Example}
\newtheorem{Algorithm}[Proposition]{Algorithm}

\begin{document}

\title[Computing the Cassels-Tate pairing]%
  {Computing the Cassels-Tate pairing on $3$-isogeny 
   Selmer groups via cubic norm equations}

\author{Monique van Beek}
\address{Canons Village 1, 46, Global edu-ro 145beon-gil, Daejeong-eup, Seogwipo-si, Jeju-do, 63644, Rep. of Korea}
\email{moniquevanbeek@gmail.com}

\author{Tom~Fisher}
\address{University of Cambridge,
         DPMMS, Centre for Mathematical Sciences,
         Wilberforce Road, Cambridge CB3 0WB, UK}
\email{T.A.Fisher@dpmms.cam.ac.uk}

\keywords{elliptic curves, Cassels-Tate pairing, descent, norm equations}
\subjclass[2010]{11G05, 11Y40}
\date{7th November 2017}  

\begin{abstract}
  We explain a method for computing the Cassels-Tate pairing on the
  $3$-isogeny Selmer groups of an elliptic curve. This improves the
  upper bound on the rank of the elliptic curve coming from a descent
  by $3$-isogeny, to that coming from a full $3$-descent. One
  ingredient of our work is a new algorithm for solving cubic norm
  equations, that avoids the need for any $S$-unit computations. As an
  application, we show that the elliptic curves with torsion subgroup
  of order $3$ and rank at least $13$, found by Eroshkin, have rank
  exactly $13$.
\end{abstract}

\maketitle

\renewcommand{\baselinestretch}{1.1}
\renewcommand{\arraystretch}{1.3}
\renewcommand{\theenumi}{\roman{enumi}}

\section*{Introduction}

Let $E$ be an elliptic curve over a number field $K$. By the
Mordell-Weil theorem, the rational points $E(K)$ form a finitely
generated abelian group. The number of points needed to generate the
non-torsion part of $E(K)$ is called the rank. Determining the rank is
a non-trivial problem, and indeed there is no known algorithm that
will compute it in all cases.

We may however bound the rank by following the proof of the
Mordell-Weil theorem. For each integer $n \ge 2$, the $n$-Selmer group
$S^{(n)}(E/K)$ classifies the $n$-coverings of $E$ that have points
everywhere locally. This group is finite and effectively computable.
Since $E(K)/n E(K)$ injects into $S^{(n)}(E/K)$, computing the
$n$-Selmer group gives an upper bound for the rank. This process is
known as (full) $n$-descent. In view of the short exact sequence
\begin{equation}
\label{exseq:ndesc}
 0 \ra E(K)/nE(K) \ra S^{(n)}(E/K) \ra \Sha(E/K)[n] \ra 0 
\end{equation}
the rank bound coming from $n$-descent may be improved whenever the
Tate-Shafarevich group $\Sha(E/K)$ contains non-trivial $n$-torsion.

Cassels~\cite{CasIV} defined an alternating pairing (now known as the
Cassels-Tate pairing) \[ \Sha(E/K) \times \Sha(E/K) \ra \Q/\Z \] with
the property that $\Sha(E/K)[n]$ and $n \Sha(E/K)$ are exact
annihilators. One consequence is that if $\Sha(E/K)$ is finite (as
conjectured by Tate and Shafarevich) then its order is a square.
Another consequence is that we can sometimes use the pairing to detect
non-trivial elements of $\Sha(E/K)$. Specifically, computing the
pairing on $S^{(n)}(E/K)$ improves the rank bound coming from
$n$-descent, to that coming from $n^2$-descent. Thus, for example,
Cassels used the pairing in \cite{cassec} to turn a $2$-descent into a
$4$-descent, and to some extent this has been generalised in
\cite{swin}, \cite{rachel}, \cite{Don}.

Descent calculations become very much more tractable in the case that
our elliptic curve admits a rational $p$-isogeny for some prime $p$.
This is the situation we consider in this paper. We write $\phi : E
\to E'$ for the $p$-isogeny, and $\phihat : E' \to E$ for its dual.
Since $\phihat \circ \phi$ is multiplication-by-$p$, there is an exact
sequence
\begin{align*}
 0   \ra & E(K)[\phi] \ra E(K)[p] \stackrel{\phi}{\ra} E'(K)[\phihat] \\
    & \ra E'(K)/ \phi E(K) \stackrel{\phihat}{\ra} E(K)/p E(K)
   \ra E(K) / \phihat E'(K) \ra 0,
\end{align*}
from which we deduce that
\[ p^{\rank E(K)} = \frac{ | E(K)/pE(K) | }{ |E(K)[p]| } = \frac{ |
  E'(K)/ \phi E(K) | \cdot | E(K)/ \phihat E'(K) | } { |E(K)[\phi] |
  \cdot |E'(K)[\phihat]| }.  \] 
Analogous to \eqref{exseq:ndesc} there are exact sequences
\[ 0 \ra E'(K)/\phi E(K)
     \ra S^{(\phi)}(E/K) \ra \Sha(E/K)[\phi_*] \ra 0, \]
and 
\[ 0 \ra E(K)/\phihat E'(K) \ra S^{(\phihat)}(E'/K) \ra
\Sha(E'/K)[\phihat_*] \ra 0. \] 
Computing the Selmer groups $S^{(\phi)}(E/K)$ and
$S^{(\phihat)}(E'/K)$ gives an upper bound for the rank. This process
is known as descent by $p$-isogeny, and is described for example in
\cite{top, delong, 5descent, ctpps, ss, flynngrattoni, millerstoll}.

There is a commutative diagram with exact rows
\[ \xymatrix{ E'(K)[\phihat] \ar@{=}[d] \ar[r] & E'(K)/ \phi E(K)
   \ar@{^{(}->}[d] \ar[r] &
  E(K)/p E(K)  \ar@{^{(}->}[d] \ar[r]  & E(K) / \phihat E'(K)  \ar@{^{(}->}[d] \\
  E'(K)[\phihat] \ar[r] & S^{(\phi)}(E/K) \ar[r] & S^{(p)}(E/K) \ar[r]
  & S^{(\phihat)}(E'/K) } \] 
where the final map in the second row need not be surjective. Instead
its image is the kernel of the Cassels-Tate pairing
\begin{equation}
\label{CTP}
 \langle~,~\rangle_\CT : S^{(\phihat)}(E'/K) \times S^{(\phihat)}(E'/K)
\to \Q/\Z.
\end{equation}
This pairing is the lift of the one on $\Sha(E'/K)[\phihat_*]$.
Computing the pairing~\eqref{CTP} allows us to turn a descent by
$p$-isogeny into a full $p$-descent.  If the pairing is non-zero then
this improves our upper bound for the rank.

The case $p=2$ is treated in \cite{higherdescent}, so from now on we
take $p$ an odd prime. In the first of his series of papers on
elliptic curves, Cassels \cite{CasI} showed how to compute the
pairing~\eqref{CTP} when $p=3$ and $E'$ takes the form $x^3 + y^3 =
k$.  The case where $p = 3$ or $5$ and $E[p] \isom \mu_p \times
\Z/p\Z$ was treated in \cite{ctpps}.

We describe a method for computing the pairing~\eqref{CTP} in the
cases where $E[\phi]$ is isomorphic (as a Galois module) to either
$\mu_p$ or $\Z/p\Z$. In both cases the global part of our method
requires us to solve a norm equation $N_{L/K}(\xi) = a$ where $L/K$ is
a field extension of degree $p$. Moreover, when $p=3$ and $E[\phi]
\isom \Z/3\Z$, we have $L = K(\sqrt[3]{b})$ for some $b \in K$. In the
case $K = \Q$ we give an algorithm for solving such norm equations,
that avoids the need for any $S$-unit computations.  This enables us
to apply our methods to elliptic curves with large discriminant.

One particular computational challenge is to find elliptic curves over
$\Q$ of large rank with a given torsion subgroup. The current records
are listed on Dujella's website \cite{dujweb}.  Between 2007 and 2009,
Y.G. Eroshkin found the following five elliptic curves with torsion
subgroup $\Z/3\Z$ and rank at least 13.
\begin{align*}
y^2 + 10154960719 x y - 66798078951809458114391930400 y &= x^3 &&17 \\
y^2 + 8412073331 x y + 7384158420201525518270114400 y &= x^3 && 13 \\
y^2 + 19223749711 x y - 435665346791890005577936749600 y &= x^3 && 13\\
y^2 + 8589423667 x y - 30679410326232604531989794400 y &= x^3 && 17\\
y^2 + 35429815349 x y - 169064164426703584254124708800 y &= x^3 && 15 
\end{align*}
The number on the right is the upper bound for the rank obtained by
descent by $3$-isogeny. By computing the Cassels-Tate pairing we were
able to verify that each of these curves has rank exactly $13$. This
is the largest known rank for an elliptic curve with torsion subgroup
$\Z/3\Z$.

We have also used the methods of this paper to find new examples of
elliptic curves with torsion subgroup $\Z/9\Z$ and ranks 3 and 4 (see
\cite{phdMonique, dujweb}). In contrast, when we searched for elliptic
curves with torsion subgroup $\Z/12\Z$ and rank~$4$ we could not find
any examples beyond the one already known. For both of these torsion
subgroups the largest known rank is $4$.

In Section~\ref{sec1} we recall the definition of the Cassels-Tate
pairing that is relevant to our work. In Section~\ref{sec2} we give an
explicit description of the long exact sequence
\begin{equation}
\label{H1seq}
 H^1(K,E[\phi]) \to H^1(K,E[p]) \to H^1(K,E'[\phihat]) \to H^2(K,E[\phi]) 
\end{equation}
in terms of \'etale algebras, and explain how lifting an element of
$H^1(K,E'[\phihat])$ to $H^1(K,E[p])$ comes down to solving a norm
equation. As indicated above, we concentrate on the cases where
$E[\phi] \isom \mu_p$ or $E[\phi] \isom \Z/p\Z$. Since it is always
possible to reduce to one of these two cases by making a field
extension of degree coprime to $p$, this is perhaps not such a severe
restriction.

In Section~\ref{sec3} we present our new algorithm for solving norm
equations for pure cubic extensions $\Q(\sqrt[3]{b})/\Q$.  It is based
on the Legendre-type method for solving conics
in~\cite{ratconics}. When applied to suitably large examples, our
algorithm performs much better than the standard approach using
$S$-units, as described for example in
\cite[Section~7.5]{cohenadvanced}, \cite{simonnorm}.

In Section~\ref{sec4} we give three examples computing the
pairing~\eqref{CTP} in the case $K = \Q$ with $E[\phi] \isom \mu_3$ or
$E[\phi] \isom \Z/3\Z$. The first two examples are small, and so do
not require any special methods to solve the norm equations.
In the third example, where we consider one of 
Eroshkin's curves with torsion subgroup $\Z/3\Z$, % and here 
we are entirely reliant on the methods in Section~\ref{sec3}.

It is interesting to remark that if we solve norm equations by the
method in Section~\ref{sec3}, then the simplest case is when $E[ \phi]
\isom \Z/3\Z$.  However, if we solve norm equations by trivialising
the corresponding cyclic algebra (using the method in
\cite{algorithms}) then the simplest case is when $E[\phi] \isom
\mu_3$. For a general $3$-isogeny we may reduce to either one of these
two simplest cases at the expense of making a quadratic extension to
our base field.

An alternative approach to improving a descent by $p$-isogeny to a
full $p$-descent is described in~\cite{creutzmiller}. The method there
does however require a rigorous computation of $S$-units in a degree
$p$ extension of $K$. In contrast, we are free to solve norm equations
by any method we like, since once a solution is found it is
straightforward to verify it is correct.

The following notation will be used throughout. For $K$ a field, we
write $\Kbar$ for its separable closure and $G_K = \Gal(\Kbar/K)$ for
its absolute Galois group. The Galois cohomology group $H^i(G_K,-)$ is
abbreviated as $H^i(K,-)$, and we write $\Hom(G_K,-)$ for the
continuous homomorphisms.  The unit group of a ring $R$ is denoted
$R^\times$.  We write $\mu_p$ for the group of $p$th roots of unity,
and $\zeta_p$ for a generator.

The calculations in Section~\ref{sec4} were carried out using Magma
\cite{magma}. This paper is based on the first author's PhD thesis
\cite{phdMonique}.

\section{The Cassels-Tate pairing} 
\label{sec1}

In this section, we define the global Cassels-Tate
pairing~\eqref{CTP}.  The definition is given as a sum of local
pairings, so we define these first.  Let $K_v$ denote the localisation
of $K$ at a place $v$. The Weil pairing $e_\phi:E[\phi]\times
E'[\phihat]\rightarrow \mu_p$ induces by cup product a pairing
\[ \cup:H^1(K_v,E[\phi])\times H^1(K_v,E'[\phihat])\rightarrow H^2(K_v,\mu_p).
\]
It terms of cocycles we have $(\xi\cup\eta)_{\sigma, \tau} =
e_\phi(\xi_\sigma, \sigma(\eta_\tau))$.  Since $H^2(K_v,\mu_p) \isom
\Br(K_v)[p]$ we can then apply the invariant map
$\text{inv}_{K_v}:\text{Br}(K_v)\rightarrow \Q/\Z$, from local class
field theory, to obtain the \emph{local Tate pairing}:
\begin{equation}
  \label{localtate}
  \langle~,~\rangle_{v,e_\phi}: H^1 (K_v, E[\phi]) 
  \times H^1(K_v,{E'}[\phihat]) \rightarrow \Q/\Z.
\end{equation}

\subsection{The global pairing}
Taking Galois cohomology of the short exact sequences
\begin{equation*}
\begin{aligned}
  \xymatrix{ 0 \ar[r] & E[\phi] \ar[r]^{\iota} \ar@{=}[d] & E[p]
    \ar@{^{(}->}[d] \ar[r]^\phi
    &   E'[\phihat] \ar@{^{(}->}[d]  \ar[r] & 0 \\
    0 \ar[r] & E[\phi] \ar[r] & E \ar[r]^\phi & E' \ar[r] & 0 }
\end{aligned}
\end{equation*}
we obtain a commutative diagram with exact rows
\begin{equation}
\begin{aligned}
\label{diag1}
\xymatrix{ H^1(K,E[p]) \ar[r]^{\phi_*} \ar[d] & H^1(K,E'[\phihat])
  \ar[r] \ar[d]
  &  H^2(K,E[\phi]) \ar[d] \\
  \prod_v H^1(K_v,E) \ar[r] & \prod_v H^1(K_v,E') \ar[r] & \prod_v
  H^2(K_v,E[\phi]) \rlap{.} }
\end{aligned}
\end{equation}

\begin{Lemma}\label{x1andx}
  Any element $x \in S^{(\phihat)}(E'/K)$ %\Sha({E'}/K)$,
  can be lifted to $x_1\in H^1(K,E[p])$ with $\phi_*(x_1)=x$.
\end{Lemma}
\begin{proof}
  The Selmer group $S^{(\phihat)}(E'/K)$ is by definition the kernel
  of the middle vertical map in~\eqref{diag1}. By a diagram chase, it
  suffices to show that the right hand vertical map is injective.
  Making a finite extension $L/K$ of degree coprime to $p$, we may
  ensure that $E[\phi]\isom\mu_p$ over $L$, and so
  $H^2(L,E[\phi])\isom \Br(L)[p]$. Let $v$ be a place of $K$. Then we
  have the following commutative diagram.
\begin{equation*}
\begin{aligned}
  \xymatrix{ H^2(K,E[\phi]) \ar[r]^-{\res} \ar[d]_-{\loc_1} &
    H^2(L,E[\phi])
    \ar[d]_-{\loc_2}\ar[r]^-{\cor} & H^2(K,E[\phi])\ar[d]_-{\loc_1}  \\
    H^2(K_v,E[\phi]) \ar[r]^-{\res}& \oplus_{w|v}H^2(L_w,E[\phi])
    \ar[r]^-{\cor} & H^2(K_v,E[\phi]) }
\end{aligned}
\end{equation*}

By global class field theory, the following exact sequence holds for
any number field $L$.
\begin{align*} %\label{BrauerShort}
   0\longrightarrow \Br(L) \longrightarrow \bigoplus_w
   \Br(L_w) \xrightarrow{\sum\inv_w} \Q/\Z
   \longrightarrow 0
\end{align*}
Thus the map $\prod_v \loc_2$ is injective.  By \cite[Proposition
3.3.7]{csagc} the composite $\cor \circ \res$ is multiplication by $n
= [L:K]$. The kernel of $\prod_v \loc_1$ is now both $p$-torsion and
$n$-torsion. Since $p$ and $n$ are coprime, it follows that $\prod_v
\loc_1$ is injective as required.
\end{proof}

The Kummer exact sequences for $[p] : E \to E$ and $\phihat : E' \to
E$ give the rows of the following commutative diagram
\begin{equation}
\label{keydiagram}
\begin{aligned}
  \xymatrix{
    & & H^1(K,E[\phi]) \ar[d]^{\iota_*} \\
    E(K) \ar[r]^-{p} \ar[d]_\phi & E(K) \ar@{=}[d]
    \ar[r]^-{\delta_p} & H^1(K,E[p]) \ar[d]^{\phi_*} \\
    E'(K) \ar[r]^-{\phihat} & E(K) \ar[r]^-{\delta_{\phihat}} &
    H^1(K,E'[\phihat]) \rlap{.} }
\end{aligned}
\end{equation}
The right hand column is the long exact sequence~\eqref{H1seq}.  We
also consider the analogue of this diagram with $K$ replaced by $K_v$.
In the terminology of \cite{PoonenStoll}, the following is the ``Weil
pairing definition'' of the Cassels-Tate pairing.

\begin{Definition}
  [Definition of the Cassels-Tate pairing]
  \label{wpDefCT}
  Let $x,y \in S^{(\phihat)}(E'/K)$. By Lemma~\ref{x1andx} there
  exists $x_1 \in H^1(K,E[p])$ with $\phi_*(x_1) = x$.  We write $x_v,
  y_v, x_{1,v}$ for the localisations of $x,y, x_1$ at a place $v$.
  For each place $v$ we pick $P_v \in E(K_v)$ with
  $\delta_{\phihat}(P_v) = x_v$. Then
% $\delta_p(P_v) - x_{1,v} = \iota_*(\xi_v)$ 
% Sign switched here to agree with the examples!
  $x_{1,v} - \delta_p(P_v)  = \iota_*(\xi_v)$ 
 for some $\xi_v \in H^1(K_v,E[\phi])$.  The
  Cassels-Tate pairing is defined as
  \begin{align*}%\label{ctpair}
   \langle x,y\rangle_{\CT} = \sum_{v} %  \in M_K} 
   \langle \xi_v,y_v\rangle_{v,e_\phi}
\end{align*}
where the sum is over all places $v$ of $K$, and
$\langle\;,\;\rangle_{v,e_\phi}$ is the local Tate
pairing~\eqref{localtate}.
\end{Definition}
It may be shown that the pairing is independent of the choice of
global lift~$x_1$ and the choices of local points~$P_v$.  For further
details, and properties of the pairing, see for example~\cite{CasIV,
MilneADT, McCallum, PoonenStoll, ctpps}.

As we describe in the next section, the local Tate pairing is closely
related to the Hilbert norm residue symbol. It can therefore be
computed using standard techniques. A more serious problem is that of
computing a global lift $x_1$ of $x$.  In Section~\ref{sec2} we
explain how this may be reduced to solving a norm equation.  This then
motivates our work on norm equations in Section~\ref{sec3}.

\subsection{Computing the local pairing}\label{computingLocalPairing}
Let $p$ be a prime. Let $\phi : E \to E'$ be a $p$-isogeny of elliptic
curves defined over a number field $K$. We fix $L/K$ a finite Galois
extension of degree coprime to $p$, such that all points in the
kernels of $\phi$ and $\phihat$ are defined over $L$. By properties of
the Weil pairing we have $\mu_p \subset L$. We may therefore fix
isomorphisms $E[\phi] \isom \mu_p$ and $E'[\phihat] \isom \mu_p$ over
$L$. These maps induce, by restriction and the Kummer isomorphism,
injective group homomorphisms
\[ \overline{w}_\phi:H^1(K,E[\phi]) \ra L^\times/(L^\times)^p, \]
and 
\[ \overline{w}_\phihat:H^1(K,E'[\phihat]) \ra L^\times/(L^\times)^p. \] 
We also write $\overline{w}_\phi$ and $\overline{w}_\phihat$ for the
local analogues of these maps.

\begin{Lemma}\label{prop:localpair}
  There exists a primitive $p$th root of unity $\zeta_p \in L$ such
  that for all places $v$ of $K$ the local Tate
  pairing~\eqref{localtate} is given by
\begin{equation}
  \label{localformula}
  \langle x, y\rangle_{v,e_\phi} = \frac{1}{[L_w:K_v]}
  \Ind_{\zeta_p}( \overline{w}_\phi(x), \overline{w}_\phihat(y))_w 
\end{equation}
where $w$ is any place of $L$ dividing $v$,
\[ (~,~)_w : L_w^\times/(L_w^\times)^p \times
L_w^\times/(L_w^\times)^p \to \mu_p \] 
is the Hilbert norm residue symbol, and $\Ind_{\zeta_p} : \mu_p \to
\frac{1}{p}\Z/\Z$ is the isomorphism sending $\zeta_p \mapsto
\frac{1}{p}$.
\end{Lemma}
\begin{proof}
  We first treat the case $L=K$. The pairing
  \begin{equation}
  \label{mup-pair}
  \mu_p \times \mu_p \to \mu_p; \, (\zeta_p^a,\zeta_p^b) \mapsto \zeta_p^{ab}
\end{equation}
induces by cup product, the Kummer isomorphism and the local invariant
map, a pairing
\[ \{ ~, ~ \}_v : K_v^\times/(K_v^\times)^p \times 
K_v^\times/(K_v^\times)^p \to \tfrac{1}{p} \Z/ \Z. \] 
The Hilbert norm residue symbol is $(x,y)_v = \zeta_p^{p \{x,y \}_v}$.
By \cite[Prop. XIV.2.6]{serre} it is independent of the choice of
$\zeta_p$.

If we make an appropriate choice of $\zeta_p$ then our identifications
$E[\phi] \isom \mu_p$ and $E'[\phihat] \isom \mu_p$ identify the Weil
pairing $e_\phi : E[\phi] \times E'[\phihat] \to \mu_p$
with~\eqref{mup-pair}. This proves~\eqref{localformula}.  The general
case, with $L \not= K$, follows by standard properties of the cup
product and the local invariant map under restriction, for which we
refer to \cite[Proposition IV.7.9(iii)]{casfroh} and \cite[Theorem
VI.1.3]{casfroh}.
\end{proof}

\begin{Remark}
  In practice we are happy to compute the Cassels-Tate pairing up to
  an overall scaling. Therefore the choice of $\zeta_p$ in
  Lemma~\ref{prop:localpair} does not matter, provided that the same
  global choice is used in all our local calculations.
\end{Remark}

We now suppose that $\mu_p \subset K_v$ and describe some methods for
computing the Hilbert norm residue symbol. In fact the symbol may be
defined with $p$ replaced by any integer $m \ge 2$, and we now work in
this generality.

\begin{Proposition}\label{UsingCProp}
  Assume that $\mu_m \subset K_v$. The Hilbert norm residue symbol
 \[  (~,~)_v : K_v^\times/(K_v^\times)^m \times  K_v^\times/(K_v^\times)^m
 \to \mu_m \]
  has the following properties.
 \begin{enumerate}
  \item $(a,b)_v(a,c)_v = (a,bc)_v$ 
% \item $(a,b)_v=1$ if either $a$ or $b\in (K_v^\times)^m$.
% ---------already implicit in the above
  \item $(a,b)_v=1$ if $b$ is a norm for the extension $K_v(\sqrt[m]{a})/K_v$.
In particular $(a,-a)_v=(a,1-a)_v=1$.
%  \item $(a,b)_v=1$ if $a+b\in (K_v^\times)^m$
% ---------follows very easily from other statement
  \item $(a,b)_v(b,a)_v=1$
  \item If $v = \pp$ is a prime not dividing $m$ and
$v_\pp(a)= 0$ then 
\[  (a,b)_v = \left(\frac{a}{\pp}\right)^{v_\pp(b)} 
\quad \text{ where } \quad
\left(\frac{a}{\pp}\right) \equiv a^{\frac{N\pp-1}{m}} \pmod{\pp}. \]
 \end{enumerate}
\end{Proposition}
\begin{proof}
 See \cite[Exercise 2]{casfroh} or \cite[Proposition II.7.1.1]{ggras}.
\end{proof}

Proposition~\ref{UsingCProp} can be used to compute $(a,b)_v$ whenever
$v\nmid m \infty$. Taking $m=p$ an odd prime, the following will
suffice for our purposes in the case $v \mid p$.  Let
$K=\Q(\zeta_p)$. Then %$p$ is totally ramified in $K/\Q$ and
$\lambda=1-\zeta_p$ generates the
% prime ideal corresponding to the unique prime $v$ 
unique prime of $K$ lying over $p$.  It is shown in
\cite[Exercise~2.13]{casfroh} that $K_v^\times/(K_v^\times)^p$ has
basis $\lambda,\eta_1 ,\ldots, \eta_p$ where $\eta_i=1-\lambda^i$, and
an explicit recipe is given for computing the Hilbert norm residue
symbol. In the case $p = 3$ this works out as
\begin{equation}
\label{table:pIs3Table}
\begin{array}{c|cccc} 
  &\lambda&\eta_1&\eta_2&\eta_3\\
  \hline
  \lambda&0&0&0&\zeta_3^{2}\\
  \eta_1&0&0&\zeta_3&0\\
  \eta_2&0&\zeta_3^{2}&0&0\\
  \eta_3&\zeta_3&0&0&0
% \end{tabular}
% \caption{Matrix for computing local pairing in $p=3$ case.}
% \end{table}
\end{array}
\end{equation}

\section{Galois cohomology}
\label{sec2}
In this section, we give an explicit description of the long exact
sequence~\eqref{H1seq} in terms of \'etale algebras. As explained in
Section~\ref{sec1} this will enable us to compute the Cassels-Tate
pairing.

\subsection{Etale algebras}
\label{sec21}
Following~\cite{ss} we interpret the Galois cohomology groups
in~\eqref{H1seq} in terms of \'etale algebras. This makes the groups
more amenable for practical computation.  We work over a field $K$ of
characteristic $0$.

Let $\Phi$ be a finite set with $G_K$-action. The \'etale algebra $D$
associated to $\Phi$ is the set of all $G_K$-equivariant maps $\Phi
\to \Kbar$. This is a $K$-algebra under pointwise operations.  If
$P_1, \ldots, P_n \in \Phi$ are representatives for the $G_K$-orbits
then evaluation at these points gives an isomorphism \[D \isom K(P_1)
\times \ldots \times K(P_n).\] In particular $D$ is a product of
finite field extensions of $K$.  We also write
$\overline{D}=D\otimes_K \overline{K}$.  This is the $\Kbar$-algebra
of all maps $\Phi \to \Kbar$.

We fix $p$ an odd prime.  Let $\psi : E \to E'$ be an isogeny of
elliptic curves with $E[\psi] \subset E[p]$, and let $\psihat$ be its
dual. Let $D$ be the \'etale algebra of $E'[\psihat]$.  Let
\begin{align*}
w_\psi : E[\psi]&\longrightarrow \mu_p(\overline{D})\\
P&\longmapsto(Q\mapsto e_\psi (P,Q)),
\end{align*}
be the map induced by the Weil pairing $e_\psi$.  This induces a map
on $H^1$'s that on composing with the Kummer isomorphism gives a group
homomorphism
\[ \overline{w}_\psi:H^1(K,E[\psi]) \longrightarrow
D^\times/(D^\times)^p. \]

Now let $\phi : E \to E'$ be a $p$-isogeny. The Weil pairings
$e_\phi$, $e_p$ and $e_{\phihat}$ are compatible in the sense that
they give an isomorphism between the exact sequence
\begin{equation}
\label{ses}
0 \longrightarrow E[\phi] \stackrel{\iota}{\longrightarrow}  E[p] \stackrel{\phi}{\longrightarrow} E'[\phihat] \longrightarrow 0. 
\end{equation}
and the exact sequence of Cartier duals.  Let $A_1$, $A_2$ and $A$ be
the \'etale algebras of $E'[\phihat]$, $E[\phi]$ and $E[p]$. By the
compatibility of the Weil pairings, % mentioned above,
we obtain a commutative diagram
\begin{equation}\label{GeneralDiagram}
\begin{aligned}
  \xymatrix{
    H^1(K,E[\phi])\ar[d]_{\overline{w}_\phi} \ar[r]^ {\iota_{*}}& 
    H^1(K,E[p]) \ar[d]_{\overline{w}_p} \ar[r]^ {\phi_{*}} & 
    H^1(K,E'[\phihat]) \ar[d]_{\overline{w}_\phihat} \\
    A_1^\times/(A_1^\times)^p \ar[r]^-{\phi^*} & A^\times/(A^\times)^p
    \ar[r]^-{\iota^*} & A_2^\times/(A_2^\times)^p }
\end{aligned}
\end{equation}
where the first row is~\eqref{H1seq}, i.e. the long exact sequence
associated to~\eqref{ses}. The maps in the second row (which is not
exact) are the pull backs by $\phi$ and $\iota$.

It is shown in \cite[Section 5]{ss} that the vertical maps
in~\eqref{GeneralDiagram} are injective, and their images are
described as follows.  We fix $g$ a primitive root mod $p$, and let
$\sigma_g$ be the automorphism of $A_1$, $A_2$ or $A$ given by
$(\sigma_g \alpha)(P) = \alpha(gP)$.  By \cite[Lemma~5.2]{ss} we have
\begin{equation}
\label{H1isog} 
H^1(K,E[\phi])\isom \text{ker}(g-\sigma_g: A_1^\times/(A_1^\times)^p \rightarrow A_1^\times/(A_1^\times)^p)
\end{equation}
and likewise for $H^1(K,E'[\phihat])$.  The corresponding description
of $H^1(K,E[p])$ involves the set $\Lambda$ of affine lines in $E[p]$
that do not pass through the origin ${\mathcal O}$.  Let $B$ be the
\'etale algebra of $\Lambda$. Then the map
\[ u: \mu_p(\overline{A}) \to \mu_p(\overline{B}) ; \,\, \alpha
\mapsto \big( \ell \mapsto \prod_{P\in \ell} \alpha(P) \big) \]
induces a map on $H^1$'s, and so by the Kummer isomorphism gives us a
group homomorphism $\overline{u} : 
A^\times/(A^\times)^p \longrightarrow B^\times/(B^\times)^p.$

\begin{Theorem}[Schaefer-Stoll {\cite[Corollary 5.9]{ss}}]
\label{TheoremSS} 
\[ 
H^1(K,E[p]) \isom \ker(g-\sigma_g: A^\times/(A^\times)^p \rightarrow
A^\times/(A^\times)^p) \cap \ker(\overline{u}).
\]  
\end{Theorem}
In fact we have $A \isom K \times A'$ where $A'$ is the \'etale
algebra of $E[p] \setminus \{ \mathcal O \}$, and in
Theorem~\ref{TheoremSS} we are free to replace $A$ by $A'$. This
description of $H^1(K,E[p])$ can sometimes be simplified using the
following lemma.

\begin{Lemma}
\label{algebra-lemma}
Let $\alpha \in \im({\overline{w}_p}) \subset A^\times/(A^\times)^p$.
If $S,T \in E[p]$ then
\[ \frac{\alpha(S) \alpha(T)}{\alpha(S+T)} \in (K(S,T)^\times)^p. \]
\end{Lemma}
\begin{proof}
See \cite[Lemma 3.8]{algebra}. 
\end{proof}

Let $E/K$ be an elliptic curve and $G$ the image of the mod $p$ Galois
representation $\overline{\rho}_{E,p} : G_K \to \GL(E[p])$. Fixing a
basis $S,T$ for $E[p]$, this is a subgroup of $\GL_2(\Z/p\Z)$. In the
next two sections we consider two specific possibilities for $G$,
which we call the {\em $\mu_p$-nonsplit} and {\em $\Z/p\Z$-nonsplit}
cases. The case where $E[p]$ splits as $\mu_p \times \Z/p\Z$ is
significantly easier, as described in~\cite{McCallum}, \cite{ctpps}.

\subsection{$\mu_p$-nonsplit case}
\label{sec:mup}
We consider $E/K$ an elliptic curve whose mod~$p$ Galois
representation has image
\[ G = \left\{ \begin{pmatrix} * & * \\ 0 & 1 \end{pmatrix} \right\}
\subset \GL_2(\Z/p\Z) \] generated by $\sigma =
(\begin{smallmatrix}1&1\\0&1\end{smallmatrix})$ and $\tau =
(\begin{smallmatrix}g&0\\0&1\end{smallmatrix})$ where $g$ is a
primitive root mod $p$. Thus we have a basis $S,T$ of $E[p]$ such that
\begin{align*}
\sigma(S)&=S&\tau(S)&=gS\\
\sigma(T)&=S+T & \tau(T)&=T.
\end{align*}
Note that $\tau\sigma=\sigma^g\tau$. Let $L_1=K(S)=K(\zeta_p)$,
$L_2=K(T)$ and $M=K(E[p])$.
\begin{align*}
  \xymatrix{
    &&M\ar@{-}[ld]_{\langle\sigma\rangle}\ar@{-}[rd]^{\langle\tau\rangle}&& \\
    \ar@{-}@/^/[u]^{C_p}& L_1 \ar@{-}[rd]_{p-1} & & L_2\ar@{-}[ld]^p&\\
    && K && \ar@{-}@/_/[uu]_{G} }
\end{align*}

Recall that $p$ is an odd prime.  Let $\chi_\cyc : G_K \to
(\Z/p\Z)^\times$ be the cyclotomic character.  Then for $N$ a
$(\Z/p\Z)[G_K]$-module we write $N^{(i)}$ for the eigenspace where
$G_K$ acts as $\chi_\cyc^i$.

\begin{Theorem}\label{explicitmu3} We have $H^1(K,E[p]) \isom H$,
  where $H$ is the group of pairs $(a,b)\in
  (L_1^\times/(L_1^\times)^p)^{(1)} \times L_2^\times/(L_2^\times)^p$
  satisfying $N_{L_2/K}(b)\in (K^\times)^p$ and $\sigma(b)/(a b) \in
  (M^\times)^p$.
\end{Theorem}
\begin{proof}
  We use the description of $H^1(K,E[p])$ in Theorem~\ref{TheoremSS}
  as the intersection of $\ker(g-\sigma_g)$ and $\ker(\overline{u})$.

  There are $p$ orbits for the action of $G_K$ on $E[p] \setminus \{
  \mathcal O \} $, with representatives $S$ and $iT$ for
  $i\in\{1,\ldots, p-1\}$.  Therefore $H^1(K,E[p])\subset
  A'^\times/(A'^\times)^p$ where
  \[ A'\isom L_1\times\underbrace{ L_2\times \ldots \times
    L_2}_{p-1}. \] Moreover $(a,b_1,b_2,\ldots, b_{p-1}) \in
  A'^\times/(A'^\times)^p$ belongs to $\ker(g-\sigma_g)$ if and only
  if $a \in (L_1^\times/(L_1^\times)^p)^{(1)}$ and $b_i \equiv b_1^i
  \mod{(L_2^\times)^p}$ for all $i\in\{1,\ldots, p-1\}$.  Accordingly
  we represent elements of $H^1(K,E[p])$ as pairs $(a,b)$ where $b =
  b_1$.

  We consider the action of $G_K$ on the set $\Lambda$ of affine lines
  in $E[p]$ missing the origin.  There are $p-1$ orbits of one line
  each, given by $\ell_1, \ldots, \ell_{p-1}$ where
  \[ \ell_i = \{ iT, S+iT, 2S + iT, \ldots, (p-1)S + iT \}, \] and
  just one further orbit of size $p^2-p$ represented by
  \[ m = \{ S, S+T, S+2T, \ldots, S+(p-1)T \}. \] Thus the \'etale
  algebra $B$ associated to $\Lambda$ is given by
  \[B \isom \underbrace{K\times\cdots\times K}_{p-1}\times M.\]

  A pair $(a,b)$ corresponding to $\alpha \in A^\times$ represents an
  element in $\ker(\overline{u})$ if and only if
  \[ N_{L_2/K} (b)^i \equiv \prod_{P \in \ell_i} \alpha(P) \equiv 1
  \mod{(K^\times)^p} \] for all $i \in \{1,2,\ldots,p-1\}$, and
  \begin{equation}
    \label{cond2}
    a \prod_{i=1}^{p-1} \sigma^{i^{-1}} (b)^{i} \equiv 
    \prod_{P \in m} \alpha(P) \equiv 1 \mod{(M^\times)^p} 
  \end{equation}
  where inverses are taken in $(\Z/p\Z)^\times$. This proves the
  theorem when $p=3$.

  In general~\eqref{cond2} may be simplified as follows.  First
  Lemma~\ref{algebra-lemma} tells us that for an element in the image
  of $H^1(K,E[p])$ we have
  \begin{equation}
    \label{cond:cfoss}
    \frac{\sigma(b)}{ab} = \frac{\alpha(S+T)}{\alpha(S) \alpha(T)}
    \in (M^\times)^p. 
  \end{equation}
  Conversely, if we assume~\eqref{cond:cfoss} then~\eqref{cond2}
  follows by an easy calculation.
\end{proof}

Let $\phi : E \to E'$ be the isogeny with kernel generated by $S$.
The first row of the following diagram is the long exact
sequence~\eqref{H1seq}. Since $E[\phi] \isom \mu_p$ as Galois modules,
we have $H^1(K,E[\phi]) \isom K^\times/(K^\times)^p$ and
$H^2(K,E[\phi]) \isom \Br(K)[p]$. The other two vertical maps are
given by Theorem~\ref{explicitmu3} and~\eqref{H1isog}.
\begin{equation}\label{importantDiagram}
\begin{aligned}
  \xymatrix{
    H^1(K,E[\phi])\ar[d]^-{\isom}\ar[r]^ {\iota_{*}}& 
    H^1(K,E[p]) \ar[d]^{\isom} \ar[r]^ {\phi_{*}} & 
    H^1(K,E'[\phihat]) \ar[d]^{\isom} \ar[r]^-{\delta_2} & 
    H^2(K,E[\phi])\ar[d]^{\isom}   \\
    K^\times/(K^\times)^p \ar[r]^-{f} & H \ar[r]^-{g} &
    (L_1^\times/(L_1^\times)^p)^{(1)} \ar[r]^-{\Delta} & \Br(K)[p] }
\end{aligned}
\end{equation}
The maps $f$, $g$ and $\Delta$ are defined so that this diagram
commutes. We now describe these maps explicitly.
\begin{Lemma}
\label{lem:fg}
We have $f:b\mapsto (1,b)$ and $g:(a,b) \mapsto a$.
\end{Lemma}
\begin{proof}
  The second rows in~\eqref{GeneralDiagram}
  and~\eqref{importantDiagram} differ in that we have applied
  projection maps
  \begin{align*}
    A_1 & \to K & A & \to L_1 \times L_2 & A_2 & \to L_1 \\
    \alpha & \mapsto \alpha(\phi T) & \alpha & \mapsto
    (\alpha(S),\alpha(T)) & \alpha & \mapsto \alpha(S)
  \end{align*}
  From this it is easy to see that the maps $f$ and $g$ in the
  statement of the lemma do indeed correspond to pull back by $\phi$
  and $\iota$.
\end{proof}

Our description of $\Delta$ will be in terms of cyclic algebras, so we
introduce these first.  Let $\chi \in \Hom(G_K,\Z/p\Z)$ and $b \in
K^\times$. If $\chi$ is non-trivial then it factors via an isomorphism
$\Gal(L/K) \isom \Z/p\Z$; $\gamma \mapsto 1$, for some degree $p$
cyclic extension $L/K$. The cyclic algebra $A = A(\chi,b)$ is the
$K$-algebra $\{ \sum_{i=0}^{p-1} a_i v^i : a_i \in L \}$ with
multiplication determined by $v^p = b$ and
\begin{equation}
\label{gammaprop}
vx = \gamma(x)v
\end{equation} for all $x \in L$. % and $v^p = b$. 
This is a central simple algebra of dimension $p^2$. We write
$(\chi,b)$ for its class in $\Br(K)$. This construction is compatible
with the cup product, in the sense that the following diagram commutes.
\begin{equation}
\label{cyclic-cup}
\begin{aligned}
  \xymatrix{ H^1(K,\Z/p\Z) \ar@<-2.5em>[d]^{\isom} \times H^1(K,\mu_p)
    \ar[r]^-{\cup} \ar@<3.8em>[d]^{\isom}
    & H^2(K,\mu_p) \ar[d]^{\isom} \\
    \Hom(G_K,\Z/p\Z) \times K^\times/(K^\times)^p \ar[r]^-{(~,~)} &
    \Br(K)[p] \rlap{.}}
\end{aligned}
\end{equation}

The next two lemmas are well known. See for example \cite[Section
4.7]{csagc}. We include the proofs since they are needed for our
algorithms.
\begin{Lemma}
\label{lem:cyclic}
Let $A = A(\chi,b)$ be a cyclic algebra.  Then $A \isom \Mat_p(K)$ if
and only if $b$ is a norm for $L/K$.
\end{Lemma}
\begin{proof}
  If $x \in L$ then by~\eqref{gammaprop} we have $(x v)^p = N_{L/K}(x)
  v^p$. So if $b$ is a norm then $A(\chi,b) \isom A(\chi,1)$. We then
  have the trivialisation
  \begin{align*}
    A(\chi,1) &\isom \End_K(L) \isom \Mat_p(K) \\
    \sum a_i v^i & \mapsto (x \mapsto \sum a_i \gamma^i(x) ).
  \end{align*}

  Conversely, suppose we are given an isomorphism $\iota : A \isom
  \Mat_p(K)$. We fix a non-zero vector $e_1 \in K^p$. Since $L$ is a
  field, the map of $K$-vector spaces $L \to K^p ; \, x \mapsto
  \iota(x) e_1$ is injective. By a dimension count it is also
  surjective.  Making this identification, we now have $\iota : A
  \isom \End_K(L)$ with $\iota(x)y = xy$ for all $x,y \in L$. Let $v_1
  \in A$ with $\iota(v_1) = \gamma$. Since~\eqref{gammaprop} is
  satisfied by both $v$ and $v_1$ it follows that $\xi = v_1^{-1} v$
  commutes with every element of $L$, and hence is in $L$. Finally
  $N_{L/K}(\xi) = (v_1 \xi)^p = v^p = b$.
\end{proof}

\begin{Lemma}
  \label{lem:switch}
  Let ${\mathbb K}$ be a field containing a primitive $p$th root of
  unity $\zeta_p$. Let $a,b \in {\mathbb K}^\times$ and let $A$ be the
  ${\mathbb K}$-algebra generated by $x$ and $y$ subject to the
  relations $x^p=a$, $y^p=b$ and $xy = \zeta_p y x$.  Then the
  following are equivalent.
  \begin{enumerate}
  \item $a$ is a norm for ${\mathbb K}(\sqrt[p]{b})/{\mathbb K}$.
  \item $b$ is a norm for ${\mathbb K}(\sqrt[p]{a})/{\mathbb K}$.
  \item $A \isom \Mat_p({\mathbb K})$.
  \end{enumerate}
\end{Lemma}
\begin{proof}
  The equivalence of (ii) and (iii) is a special case of
  Lemma~\ref{lem:cyclic}.  By symmetry this also gives the equivalence
  of (i) and (iii).
\end{proof}

We also need the following fact about cup products.
\begin{Lemma}
\label{lem:deltas}
Let $0 \to A_1 \to A_2 \to \Z/p\Z \to 0$ be a short exact sequence of
$(\Z/p\Z)[G_K]$-modules. Then the connecting maps in the long exact
sequence
\[ \Z/p\Z \stackrel{\delta_1}{\longrightarrow} H^1(K,A_1) \rightarrow
H^1(K,A_2) \rightarrow H^1(K,\Z/p\Z)
\stackrel{\delta_2}{\longrightarrow} H^2(K,A_1) \] are related by
$\delta_2(b) = \delta_1(1) \cup b$.
\end{Lemma}
\begin{proof} Let $b$ be represented by a cocycle $(b_\sigma)$.  Let
  $x \in A_2$ with $x \mapsto 1$. Then
  \[ \delta_2(b)_{\sigma \tau} = \sigma (b_\tau x) - b_{\sigma \tau} x
  + b_{\sigma} x = b_\tau (\sigma x - x) = (\delta_1(1) \cup
  b)_{\sigma \tau}.  \] 
  Alternatively, this is \cite[Proposition 3.4.8]{csagc} with $A_3 = B
  = \Z/p\Z$.
\end{proof}

We are now ready to describe the map $\Delta$
in~\eqref{importantDiagram}.  We have $E'(K)[\phihat] \isom \Z/p\Z$
generated by $\phi(T)$. The image of $\phi(T)$ under the connecting
map in the long exact sequence associated to~\eqref{ses} is an element
$\beta \in H^1(K,E[\phi]) \isom K^\times/(K^\times)^p$.  This is a
Kummer generator for the extension $M/L_1$.  We write $\chi_a \mapsto
a$ for the natural isomorphism (depending on a choice of primitive
$p$th root of unity)
\begin{equation}
\label{kum:ab}
  \Hom(G_K,\Z/p\Z) \isom (L_1^\times/(L_1^\times)^p)^{(1)}. 
\end{equation}
By Lemma~\ref{lem:deltas} and~\eqref{cyclic-cup}, the map $\Delta$
in~\eqref{importantDiagram} is given by $\Delta : a \mapsto (\chi_a,
\beta)$, at least up to multiplication by a fixed element of
$(\Z/p\Z)^\times$, which we have no need to make explicit.

Since the diagram~\eqref{importantDiagram} commutes, and the first row
is exact, the second row is also exact.  In particular, if $a \in
(L_1^\times/(L_1^\times)^p)^{(1)}$ with $\Delta(a) = 0$ then there
exists $b \in L_2^\times/(L_2^\times)^p$ such that $(a,b) \in H$.  We
show that making the lift from $H^1(K,E'[\phihat])$ to $H^1(K,E[p])$
explicit comes down to solving a norm equation.

\begin{Theorem}\label{lemmu3}
  Let $E/K$ be an elliptic curve with $p$-torsion of type
  $\mu_p$-nonsplit.  Let $a\in (L_1^\times /(L_1^\times)^p)^{(1)}$. If
  $\Delta(a)=0$ then there exists $\xi \in M$ satisfying
  $N_{M/L_1}(\xi)=a$, and we may lift $a$ to $(a,b) \in H$ where
  \[ b=N_{M/L_2} \left( \prod_{i=1}^{p-1} \sigma^i(\xi)^i \right). \]
\end{Theorem}
\begin{proof}
  Let $\chi_a \mapsto a$ under the isomorphism~\eqref{kum:ab}, and let
  $F$ be the fixed field of the kernel of $\chi_a$.  In other words,
  $F/K$ is the degree $p$ subextension of $L_1(\sqrt[p]{a})/K$.
\begin{align*}
  \xymatrix{
    &L_1(\sqrt[p]{a}) \ar@{-}[ld]_{p}\ar@{-}[d]^{p-1} \\
    L_1 = K(\zeta_p) \ar@{-}[d]_{p-1}  & F\ar@{-}[ld]^p \\
    K & }
\end{align*}
If $\Delta(a)=0$ then Lemma~\ref{lem:cyclic} tells us that $\beta$ is
a norm for $F/K$, and hence for $L_1(\sqrt[p]{a})/L_1$. It follows by
Lemma~\ref{lem:switch} that $a$ is a norm for $M/L_1$.

Now let $\xi$ and $b$ be as in the statement of the theorem.  We show
that $(a,b) \in H$ by checking the conditions in
Theorem~\ref{explicitmu3}. First we compute
\[ N_{L_2/K} (b) = N_{M/K} \left( \prod_{i=1}^{p-1} \sigma^i(\xi)^i
\right) = N_{L_1/K}(a)^{p(p-1)/2}. \] Since $p$ is odd this gives
$N_{L_2/K}(b) \in (K^\times)^p$.  Since $\tau \sigma = \sigma^g \tau$
we have
\[ \sigma(b) \equiv \prod_{r=0}^{p-2} \tau^r \prod_{i=0}^{p-1}
\sigma^{i + g^{-r}}(\xi)^i \equiv \prod_{r=0}^{p-2} \tau^r
\prod_{i=0}^{p-1} \sigma^{i}(\xi)^{i - g^{-r}} \mod{(M^\times)^p}.  \]
Then since $N_{M/L_1}(\xi) = a$ and $\tau(a) \equiv a^g
\mod{(L_1^\times)^p}$ we have
\[ b/\sigma(b) \equiv \prod_{r=0}^{p-2} \tau^r N_{M/L_1}(\xi)^{g^{-r}}
\equiv \prod_{r=0}^{p-2} (\tau^r a)^{g^{-r}} \equiv a^{-1}
\mod{(M^\times)^p}. \] Therefore $\sigma(b)/(ab) \in (M^\times)^p$ as
required.
\end{proof}

\subsection{$\Z/p\Z$-nonsplit case}
\label{sec:ZpZ}

We consider $E/K$ an elliptic curve whose mod~$p$ Galois
representation has image
\[ G = \left\{ \begin{pmatrix} 1 & * \\ 0 & * \end{pmatrix} \right\}
\subset \GL_2(\Z/p\Z) \] generated by $\sigma =
(\begin{smallmatrix}1&1\\0&1 \end{smallmatrix})$ and $\tau =
(\begin{smallmatrix}1&0\\0&g\end{smallmatrix})$ where $g$ is a
primitive root mod $p$. Thus we have a basis $S,T$ of $E[p]$ such that
\begin{align*}
\sigma(S)& = S & \tau(S) &= S\\
\sigma(T)& = S+T & \tau(T)& = gT
\end{align*}
Note that $\tau \sigma^g = \sigma \tau$. Let $L_1 = K(\zeta_p)$ and
$M=K(T)=K(E[p])$. Let $L_2$ be the subfield of $M$ fixed by $\tau$.
The diagram of fields is the same as in Section~\ref{sec:mup}.

\begin{Theorem}\label{copycatTheorem}
  We have $H^1(K,E[p])\isom H$, where $H$ is the group of pairs
  $(a,b)\in K^\times/(K^\times)^p \times M^\times/(M^\times)^p$
  satisfying $b^g/\tau(b) \in (M^\times)^p$,
  $N_{M/L_1}(b) %= b\sigma(b)\ldots\sigma^{p-1}(b)
  \in(L_1^\times)^p$ and $\sigma(b)/(a b) \in (M^\times)^p$.
\end{Theorem}
\begin{proof}
  Again we use the description of $H^1(K,E[p])$ in
  Theorem~\ref{TheoremSS} as the intersection of $\ker(g-\sigma_g)$
  and $\ker(\overline{u})$.

  There are $p$ orbits for the action of $G_K$ on $E[p]\setminus
  \{\mathcal{O}\}$, with representatives $iS$ for $i\in\{1,\ldots,
  p-1\}$, and $T$. Therefore $H^1(K,E[p])\subset
  A'^\times/(A'^\times)^p$ where
  \begin{align*}
    A' \isom \underbrace{K \times \cdots \times K}_{p-1} \times M.
  \end{align*}
  Moreover $(a_1,a_2,\ldots, a_{p-1},b) \in A'^\times/(A'^\times)^p$
  belongs to $\ker(g-\sigma_g)$ if and only if $a_i \equiv a_1^i
  \mod{(K^\times)^p}$ for all $i\in\{1,\ldots, p-1\}$, and
  $b^g/\tau(b) \in (M^\times)^p$.  Accordingly we represent elements
  of $H^1(K,E[p])$ as pairs $(a,b)$ where $a = a_1$.

  We consider the action of $G_K$ on the set $\Lambda$ of affine lines
  in $E[p]$ missing the origin.  There is one orbit of size $p-1$
  represented by the line
  \[ \ell = \{ T, S+T, 2S + T, \ldots, (p-1)S + T \}, \] and $p-1$
  orbits of size $p$, represented by the lines $m_1, \ldots, m_{p-1}$
  where
  \[ m_i = \{ iS, iS + T, iS + 2T, \ldots, iS + (p-1)T \}. \] Thus the
  \'etale algebra $B$ associated to $\Lambda$ is given by
  \begin{align*}
    B\isom L_1 \times \underbrace{L_2 \times \cdots \times L_2}_{p-1}.
  \end{align*}

  A pair $(a,b)$ corresponding to $\alpha \in A^\times$ represents an
  element in $\ker(\overline{u})$ if and only if
  \[ N_{M/L_1} (b) \equiv \prod_{P \in \ell} \alpha(P) \equiv 1
  \mod{(L_1^\times)^p} \] and
  \begin{equation}
    \label{cond2again}
    a^iN_{M/L_2}(\sigma^i(b)) \equiv 
    \prod_{P \in {m_i}} \alpha(P) \equiv 1 \mod{(L_2^\times)^p}
  \end{equation}
  for all $i \in \{1,2,\ldots,p-1\}$.

  The condition~\eqref{cond2again} may be simplified as follows.
  First Lemma~\ref{algebra-lemma} tells us that for an element in the
  image of $H^1(K,E[p])$ we have
  \begin{equation}
    \label{cond:cfoss:again}
    \frac{\sigma(b)}{ab} = \frac{\alpha(S+T)}{\alpha(S) \alpha(T)}
    \in (M^\times)^p. 
  \end{equation}
  Conversely, if we assume~\eqref{cond:cfoss:again} then $\sigma^i(b)
  \equiv a^i b \mod{(M^\times)^p}$. If in addition $b^g/\tau(b) \in
  (M^\times)^p$ then by taking norms from $M$ down to $L_2$ it follows
  that
  \[ a^i N_{M/L_2} (\sigma^i(b)) \equiv N_{M/L_2} (b)\equiv 1
  \mod{(L_2^\times)^p}. \qedhere \]
\end{proof}

Let $\phi : E \to E'$ be the isogeny with kernel generated by $S$.
The first row of the following diagram is the long exact
sequence~\eqref{H1seq}. Since $E[\phi] \isom \Z/p\Z$ over $K$, we have
$E[\phi] \isom \mu_p$ over $L_1$ and hence $H^1(L_1,E[\phi]) \isom
L_1^\times/(L_1^\times)^p$ and $H^2(L_1,E[\phi]) \isom
\Br(L_1)[p]$. The first and last vertical maps are then obtained by
the inflation-restriction exact sequence. Since $E'[\phihat] \isom
\mu_p$ over $K$ we have $H^1(K,E'[\phihat]) \isom
K^\times/(K^\times)^p$. The remaining vertical map is given by
Theorem~\ref{copycatTheorem}.
\begin{equation}\label{otherImportantDiagram}
\begin{aligned}
  \xymatrix{ H^1(K,E[\phi])\ar[d]^-{\isom}\ar[r]^ {\iota_{*}}&
    H^1(K,E[p]) \ar[d]^-{\isom} \ar[r]^ {\phi_{*}} &
    H^1(K,E'[\phihat]) \ar[d]^-{\isom} \ar[r]^{\delta_2} &
    H^2(K,E[\phi])\ar[d]^{\isom}  \\
    \,\,\, (L_1^\times/(L_1^\times)^p)^{(1)} \ar[r]^-{f} & H
    \ar[r]^-{g} & K^\times/(K^\times)^p \ar[r]^-{\Delta} &
    \Br(L_1)[p]^{(1)} \!\!\!\!  }
\end{aligned}
\end{equation}
Again we define the maps $f$, $g$ and $\Delta$ so that this diagram
commutes. We now describe these maps explicitly.
\begin{Lemma}
  We have $f:b\mapsto (1,b)$ and $g:(a,b) \mapsto a$.
\end{Lemma}
\begin{proof}
The proof is almost identical to that of Lemma~\ref{lem:fg}. 
%   The second rows in~\eqref{GeneralDiagram}
%   and~\eqref{otherImportantDiagram} differ in that we have applied
%   projection maps
% \begin{align*}
%   A_1 & \to L_1 & A & \to L_1 \times L_2 & A_2 & \to K \\
%   \alpha & \mapsto \alpha(\phi T) & \alpha & \mapsto
%   (\alpha(S),\alpha(T)) & \alpha & \mapsto \alpha(S)
% \end{align*}
% From this it is easy to see that the maps $f$ and $g$ defined in the
% statement of the lemma do indeed correspond to pull back by $\phi$ and
% $\iota$.
\end{proof}

Let $\beta$ be a Kummer generator for $M/L_1$. We write $\chi_a
\mapsto a$ for the Kummer isomorphism $H^1(L_1,\mu_p) \isom
L_1^\times/(L_1^\times)^p$.  Then exactly as in Section~\ref{sec:mup},
the map $\Delta$ is given by $\Delta : a \mapsto (\chi_a,\beta)$.
Again we show that making the lift from $H^1(K,E'[\phihat])$ to
$H^1(K,E[p])$ explicit comes down to solving a norm equation.
\begin{Theorem}\label{lemz3}
  Let $E/K$ be an elliptic curve with $p$-torsion of type
  $\Z/p\Z$-nonsplit.  Let $a\in K^\times/(K^\times)^p$.  If
  $\Delta(a)=0$ then there exists $\xi \in L_2$ satisfying
  $N_{L_2/K}(\xi)=a$, and we may lift $a$ to $(a,b) \in H$ where
  \[ b= \prod_{i=1}^{p-1} \sigma^i(\xi)^{p-i}.\]
\end{Theorem}
\begin{proof}
  If $\Delta(a)=0$ then Lemma~\ref{lem:cyclic} tell us that $\beta$ is
  a norm for $L_1(\sqrt[p]{a})/L_1$. It follows by
  Lemma~\ref{lem:switch} that $a$ is a norm for $M/L_1$, and hence for
  $L_2/K$.

  We show that $(a,b) \in H$ by checking the conditions in
  Theorem~\ref{copycatTheorem}. Since $\tau \sigma^g = \sigma \tau$
  and $\tau(\xi) = \xi$ we have
  \[ \tau(b) \equiv \prod_{i=1}^{p-1} \sigma^{g^{-1}i}(\xi)^{-i}
  \equiv \prod_{i=1}^{p-1} \sigma^{i}(\xi)^{-gi} \equiv b^{g}
  \mod{(M^\times)^p}. \] Since $\Gal(M/L_1) = \langle \sigma \rangle$
  and $p$ is odd, we have $N_{M/L_1}(b) = a^{p(p-1)/2} \in
  (K^\times)^p$, and $\sigma(b)/b \equiv N_{M/L_1}(\xi) \equiv a
  \mod{(M^\times)^p}$.
\end{proof}

\begin{Remark} It can be shown that $\beta \in
  (L_1^\times/(L_1^\times)^p)^{(2)}$. In particular, if $p=3$ then
  $\beta \in K^\times/(K^\times)^p$ and $L_2 = K(\sqrt[3]{\beta})$ is
  a pure cubic extension of $K$. Norm equations for extensions of this
  form are the subject of the next section.
\end{Remark}

\section{Solving norm equations}
\label{sec3}
In this section, we present a new algorithm for solving norm equations
in pure cubic extensions of the rationals. It is based on the
Legendre-type method for solving conics in~\cite{ratconics}.

\subsection{Diagonal cubic surfaces}
Let $K$ be a number field with ring of integers $\OK$.  Let $L =
K(\sqrt[3]{b})$ for some $b\in K$ not a cube.  We may represent any
element $\xi = A+B\sqrt[3]{b}+C\sqrt[3]{b}^2$ in $L$ in the form
\begin{align}\label{noZeroes}
  \xi = \frac{\alpha+\beta\sqrt[3]{b}}{\gamma+\delta\sqrt[3]{b}}.
\end{align}
Indeed this is clear if $B = C = 0$, and otherwise we put
\begin{align*}
\alpha &= AB-C^2b &\beta&=B^2-AC\\
\gamma&=B&\delta&=-C.
\end{align*}
Taking norms in~\eqref{noZeroes}
% We have $N_{L/K}(\xi) = (\alpha^3+b\beta^3)/(\gamma^3+b\delta^3)$,
we see that solving the norm equation $N_{L/K}(\xi) =a$ is equivalent
to find a $K$-rational point on the diagonal cubic surface
\begin{align*} 
  V_{a,b} = \{x_1^3+ax_2^3+bx_3^3+abx_4^3=0\} \subset \PP^3.
\end{align*}

\begin{Theorem}\label{swaps}
  Let $a,b\in K$. Then the following are equivalent
\begin{enumerate}
 \item $a$ is a norm for $K(\sqrt[3]{b})/K$.
 \item $b$ is a norm for $K(\sqrt[3]{a})/K$.
 \item $a^2b$ is a norm for $K(\sqrt[3]{a+b})/K$.
 \item $a^2b$ is a norm for $K(\sqrt[3]{a-b})/K$.
 \item $V_{a,b}(K)\neq \emptyset$.
\end{enumerate}
\end{Theorem}

\begin{proof}
  We may assume that $a$ and $b$ are not cubes, otherwise conditions
  (i), (ii) and (v) are trivially satisfied.  We have already shown
  that (i) and (v) are equivalent.  The symmetry in (v) also shows
  that (ii) and (v) are equivalent.

  We now prove that (ii) and (iii) are equivalent. Suppose that $b$ is
  a norm for $K(\sqrt[3]{a})/K$. Then $b/a$ is a norm for
  $K(\sqrt[3]{a})/K$, and so by the equivalence of (i) and (ii), $a$
  is a norm for $K(\sqrt[3]{b/a})/K$. But then $a+b = a(1+b/a)$ is a
  norm for $K(\sqrt[3]{b/a})/K$, and again by the equivalence of (i)
  and (ii), $b/a$ is a norm for $K(\sqrt[3]{a+b})/K$.  The converse is
  proved by reversing these steps.  The same argument, with $b$
  replaced by $-b$, shows that (ii) and (iv) are equivalent.
\end{proof}

As observed by Selmer~\cite{Selmer}, it follows from the equivalence
of (i) and (v) in Theorem~\ref{swaps}, and the Hasse norm theorem,
that the surfaces $V_{a,b}$ satisfy the Hasse principle. We now turn
this into an algorithm for solving norm equations, at least in the
case $K= \Q$.  First we record an easy lemma.

\begin{Lemma}
\label{easylemma}
Let $a,b \in \OK$. If the surface $V_{a,b}$ is locally soluble at a
prime $\pp$ and $v_\pp(b) \not\equiv 0 \pmod{3}$ then $a$ is a cube
mod $\pp$.
\end{Lemma}
\begin{proof} Working in the completion $K_\pp$ we may assume that
  $v_\pp(b) = 1$ or $2$. Let $(x_1: \ldots:x_4)$ be a local point with
  $\min v_\pp(x_i) = 0$. Then $x_1^3 + a x_2^3 \equiv 0 \pmod{\pp}$.
  If $a$ is not a cube mod $\pp$ then $x_1 \equiv x_2 \equiv 0
  \pmod{\pp}$.  But then $x_3^3 + a x_4^3 \equiv 0 \pmod{\pp}$, and we
  likewise deduce that $x_3 \equiv x_4 \equiv 0 \pmod{\pp}$. This
  contradicts that $\min v_\pp(x_i) = 0$. Therefore $a$ must be a cube
  mod $\pp$.
\end{proof}

We suppose as above that $a,b \in \OK$, and that $V_{a,b}$ is
everywhere locally soluble. If $K$ has class number $1$ then we may
assume that $(b)$ is cube-free, and indeed write $(b) = \bb_1 \bb_2^2$
where $\bb_1$ and $\bb_2$ are coprime and square-free. Then by
Lemma~\ref{easylemma} and the Chinese Remainder Theorem there exists
$c \in \OK$ such that $a \equiv c^3 \pmod{\bb_1}$. Writing $\bb_1 =
(b_1)$ for some $b_1 \in \OK$ it follows that the binary cubic form
\begin{align}\label{specialbincub}
  F(X,Y) = \frac{1}{b_1}\bigg( (cX+b_1 Y)^3-aX^3 \bigg) 
\end{align}
has coefficients in $\OK$. This form has discriminant $\Delta(F) =
-27a^2b_1^2$.

We seek to find $u,v \in \OK$, not both zero, such that $F(u,v)$ is
small. In the next section we explain how to do this in the
case $K=\Q$.

\subsection{Reduction of binary cubic forms}
\label{sec:red}
Let $G$ in $\R[X,Y]$ be a binary quadratic form, and $\Delta(G)$ its
discriminant:
\begin{align*}
  G(X,Y) &= aX^2+bXY+cY^2,\\
  \Delta(G) &= b^2-4ac.
\end{align*}
The group $\SL_2(\Z)$ acts on $\R[X,Y]$ via
\[ G(X,Y)\cdot \begin{pmatrix}\alpha&\beta\\
  \gamma&\delta\end{pmatrix} = G(\alpha X+\beta Y, \gamma X+\delta
Y)\] and the discriminant is invariant under this action.

\begin{Definition}\label{simpleReductionDef}
  A positive definite binary quadratic form $G(X,Y) = aX^2+bXY+cY^2$
  is \emph{reduced} if $|b|\leq a\leq c$.
%  with $b>0$ if there is equality. 
  Equivalently, $G$ is reduced if the root of $G(X,1)=0$ in the upper
  half plane $H$ lies in the fundamental region
 \begin{align*}
   \mathcal{F}%_\Q
   = \left \{ z \; % \mid
     \Bigg| \; z\in H, \; |z|\geq 1, \; -\frac{1}{2}\leq
     \textup{Re}(z)\leq \frac{1}{2} \right \}.
 \end{align*}
\end{Definition}
Consider now the general binary cubic form and its discriminant
\begin{align*}
  f(X,Y) &= aX^3+bX^2Y+cXY^2+dY^3 \numberthis \label{bincubic} \\
  \Delta(f)&=b^2c^2-4ac^3-4b^3d-27a^2d^2+18abcd.
\end{align*}
If $\Delta(f)<0$, then $f$ has one real root and a pair of complex
conjugate roots $\beta,\overline{\beta}$. We associate to $f$ the
binary quadratic form
\begin{align}\label{assoQ}
Q(f) = (X-\beta Y)(X-\overline{\beta} Y).
\end{align}
There are other forms we could choose, some of which are discussed in
\cite{cremonaReduction}, however this is the simplest option, and is
sufficient for our purposes.

\begin{Definition}\label{defReduced}
  A binary cubic form~\eqref{bincubic} is \emph{Minkowski-reduced} if
  the positive definite form $Q(f)$ in~\eqref{assoQ} is reduced in the
  sense of Definition~\ref{simpleReductionDef}.
\end{Definition}

We use the following result from the geometry of numbers
\cite[II.5.4]{geomnum}.

\begin{Theorem}[Davenport \cite{DavenportCubics}]
  \label{bdthm}
  If $f$ in $\Z[X,Y]$ is a binary cubic form with discriminant $\Delta
  = \Delta(f)<0$, then there are integers $(u,v) \neq (0,0)$ such that
  \begin{equation*}
  | f(u,v)| \leq \left |\frac{\Delta}{23}\right |^{1/4}.
\end{equation*}
If, further, $f$ is Minkowski-reduced in the sense of Definition
\ref{defReduced}, then
\[ \textup{min} \big\{|f(1,0)|,|f(0,1)|,|f(1,\pm 1)|,|f(1, \pm 2)|\big\} \leq
\left |\frac{\Delta}{23}\right |^{1/4}, \] with equality only when $f(X, \pm Y)
= A(X^3 + X^2Y + 2XY^2 + Y^3)$.
\end{Theorem}
The second part of the theorem, together with the well known algorithm
for reducing positive definite binary quadratic forms, gives an
algorithm for finding integers $u,v$ satisfying the conditions in the
first part of the theorem.

\subsection{An algorithm over the rationals}
We now take $K = \Q$. Let $a$ and $b$ be positive cube-free
integers. We write $b = b_1 b_2^2$ where $b_1$ and $b_2$ are positive,
coprime and square-free.  Applying the results of
Section~\ref{sec:red} to the binary cubic~\eqref{specialbincub}, we
can find $u,v \in \Z$ such that
\begin{align}\label{conds}
  0 < F(u,v) < \left(\frac{27}{23}\right)^{1/4}(a b_1)^{1/2}.
\end{align}

We observe that
\begin{equation}
\label{eqn-star}
N_{\Q(\sqrt[3]{a})/\Q}\bigg( b_2 ((cu+b_1v)-\sqrt[3]{a}u) \bigg) =
b_1b_2^3F(u,v)=bb_2F(u,v). 
\end{equation}
If we can find $\eta \in \Q(\sqrt[3]{a})$ such that
$N_{\Q(\sqrt[3]{a})/\Q}(\eta)=b_2F(u,v)$ then, by the multiplicativity
of the norm, we can find $\xi \in \Q(\sqrt[3]{a})$ such that
$N_{\Q(\sqrt[3]{a})/\Q}(\xi)=b$. Ideally, we want $b_2F(u,v)<b$, so
that our norm equation is replaced by a smaller one.  Unfortunately,
the bound~\eqref{conds} isn't quite strong enough to prove this. Our
solution to this problem is to use condition (iv) in
Theorem~\ref{swaps}.
\begin{Algorithm}\label{normAlg} (Legendre-type algorithm for solving
  cubic norm equations) \\ \textbf{Input:} A pair of positive integers
  $(a,b)$ such that $b$ is a norm
  for $\Q(\sqrt[3]{a})/\Q$.\\
  \textbf{Output:} A list of pairs $(a,b)$, with $b$ a norm for
  $\Q(\sqrt[3]{a})/\Q$, such that a solution to each norm equation
  allows us to read off a solution to the previous one.
  \begin{enumerate}
  \item \label{stepi} Replace $a$ and $b$ by their cube-free parts.
    If $a>b$ then swap $a$ and $b$.
  \item If $a=0$ or $1$ then stop.
  \item \label{stepf} Write $b = b_1 b_2^2$ where $b_1$ and $b_2$ are
    positive, coprime and square-free. Solve for $c \in \Z$ such that
    $a \equiv c^3 \pmod{b_1}$.
  \item Define $F \in \Z[X,Y]$ as in~\eqref{specialbincub}.  Use
    reduction theory to find $u,v \in \Z$ satisfying~\eqref{conds}.
  \item \label{case1} If $b_2 F(u,v) < \frac{3}{4} b$ then replace
    $(a,b)$ by $(a,b_2 F(u,v))$ and go to Step~\eqref{stepi}.
  \item \label{case2} Otherwise, replace $(a,b)$ by $(b-a,a^2 b)$ and
    go to Step~\eqref{stepi}.
  \end{enumerate}
\end{Algorithm}
When the algorithm terminates, it is clear by~\eqref{eqn-star} and the
proof of Theorem~\ref{swaps} that we may solve the original norm
equation.

\begin{Theorem}
  If $a,b \le B$ then Algorithm~\ref{normAlg} takes $O((log B)^2)$
  iterations.
\end{Theorem}
\begin{proof}
  In Step~\eqref{case1} we have $a_\textup{new} = a$ and
  $b_\textup{new} < \frac{3}{4} b$. So if we never reach
  Step~\eqref{case2} then the algorithm takes $O(\log B)$ iterations.
  If we reach Step~\eqref{case2} then
  \[ \frac{3}{4} b \le b_2 F(u,v) < \left(\frac{27}{23}\right)^{1/4}
  (a b)^{1/2} \] and so $b < 1.93 a$. In this case $a_\textup{new} =
  b- a < 0.93 a$, and so the total number of applications of
  Step~\eqref{case2} is $O(\log B)$. Moreover $b_\textup{new} = a^2 b
  < 2 a^3 \le 2 B^3$ and so Step~\eqref{case1} is applied $O(\log B)$
  times between each application of Step~\eqref{case2}.
\end{proof}

\begin{Remark}
  The bottleneck in Algorithm~\ref{normAlg} comes in
  Steps~\eqref{stepi} and~\eqref{stepf}, as these are the steps that
  involve factoring. We expect it would be possible to modify the
  algorithm, along the lines of \cite[Section 2.5]{ratconics}, so that
  factoring is only required on the first iteration. However we have
  not worked out the details.
\end{Remark}

We give two examples, the first illustrating the need for Step (vi),
and the second in preparation for Example~\ref{ThirdEx}.  The actual
solutions to the norm equations are rather large, so we do not record
them here.

\begin{Example}\label{NormEx1}
  Let $a = 5316$ and $b = 35685$. The steps taken by
  Algorithm~\ref{normAlg} are recorded in the rows of the following
  table.  On the second iteration we have $b_2 F(u,v) = 5382 > b$ and
  so we reach Step (vi).
  \[ \begin{array}{cc|ccccc} a & b & b_1 & b_2 & c & u & v \\ \hline
    5316 & 35685 & 3965 & 3 & 2521 & -11 & 7 \\
    5316 &  5364 &  149 & 6 & 52   & -2  & 1 \\
    48 & 151585867584 & \multicolumn{5}{c}{
      \text{ [take cube-free parts] } }  \\
    6 &  87723303 & 447 & 443 & 123 & -11 &  3 \\
    6 &      6202 & 6202 & 1 & 2596 & -43 &  18 \\
    6 &        77 & 77 & 1 &  41 & 2 & -1 \\
    1 & 6
  \end{array}
  \]
\end{Example}

\begin{Example}\label{NormEx2}
  Let $a = 17$ and $b = 2850760453176384635894983495759$. On the first
  iteration we have $c = 2512758208506770505416151958382$ and
  \[ (u,v) = (-1056910260262351,931597016217248). \] On this and
  subsequent iterations we have $b_1 = b$ and $b_2=1$.
  \[ \begin{array}{cc|ccc} a & b & c & u & v \\ \hline
    17 &  3227115996467513 & 3079766255214306 & 1678826 & -1602171 \\
    17 &  69326065         & 67724958         & -3767   & 3680 \\
    17 & 13311 & % = 3^3 \cdot 493 &
    \multicolumn{3}{c}{
      \text{ ~\quad [take cube-free parts] } }  \\
    17 &  493              & 476              & -1      & 1 \\
    10 &  17               & 3                & 1       &  0 \\
    1 & 10
\end{array}
\]
\end{Example}

In \cite[Chapter 4]{phdMonique} we investigated analogues of
Algorithm~\ref{normAlg} over other number fields with small
discriminant. Although we couldn't prove that these methods always
work, they seem to perform quite well in practice, at least in
reducing the norm equations to ones that can be solved by traditional
methods.

\section{Examples}
\label{sec4}
In this section, we give some examples in the case $K= \Q$, showing
how the results of Sections~\ref{sec2} and~\ref{sec3} may be used to
compute the Cassels-Tate pairing on $3$-isogeny Selmer groups.
Further examples are given in \cite{phdMonique}.

We identify $\frac{1}{3}\Z/\Z$ with $\Z/3\Z$ via multiplication by
$3$, so that the matrices below have entries $0,1,2$ rather than
$0,\frac{1}{3},\frac{2}{3}$. We also write $\langle a_1, a_2 \ldots
\rangle$ for the subgroup generated by $a_1, a_2, \ldots$.

\begin{Example}\label{mu3FirstEx}
  Let $E$ and $E'$ be the $3$-isogenous elliptic curves 
  labelled 63531c1 and 63531c2 in Cremona's tables~\cite{cremdata}.
  \begin{align*}
  E: \quad y^2 &=x^3-3(4x+52)^2 \\
  E': \quad y^2 & =x^3+36^2(x+543)^2 
  \end{align*}
  The Galois action on $E[3]$ is of type $\mu_3$-nonsplit.  Indeed
  $E[3]$ is generated by
  \[ S=(0,52\sqrt{-3}) \quad \text{ and } \quad T =
  (156/(\theta-4),156 \theta/(\theta-4)) \] where $\theta =
  \sqrt[3]{181}$.  We set $\zeta_3 = (-1 + \sqrt{-3})/2$.  As in
  Section~\ref{sec:mup} we have fields $L_1=\Q(\zeta_3)$,
  $L_2=\Q(\theta)$ and $M=\Q(\zeta_3,\theta)$. A descent by
  $3$-isogeny (see the introduction for references) computes the
  Selmer groups
  \begin{align*}
    S^{(\phi)}(E/\Q) &= \langle 181\rangle \subset \Q^\times/(\Q^\times)^3 \\
    S^{(\phihat)}({E'}/\Q) &= \langle \zeta_3,39\zeta_3 + 52 \rangle
    \subset \left(L_1^\times/(L_1^\times)^3\right)^{(1)}.
  \end{align*}
  This gives an upper bound of $2$ for the rank of $E(\Q)$.  We seek
  to improve this bound by computing the Cassels-Tate pairing on
  $S^{(\phihat)}({E'}/\Q)$.

  We start by lifting $a_1 = \zeta_3$ and $a_2 = 39\zeta_3+52$
  globally to $H^1(\Q,E[3]) \isom H$, where $H\subset
  L_1^\times/(L_1^\times)^3 \times L_2^\times/(L_2^\times)^3$ is given
  by Theorem \ref{explicitmu3}.  We used the existing function {\tt
    NormEquation} in Magma (this example is too small for the methods
  of Section~\ref{sec3} to be needed) to solve the norm equations
  $N_{M/L_1}(\xi) = a_i$ for $i=1,2$, and then computed $b_i$ with
  $(a_i,b_i) \in H$ using Theorem~\ref{lemmu3}.  We used the method
  in~\cite[Section 2]{improve4} to find a small representative for
  $b_i$ in $L_2^\times/(L_2^\times)^3$.  By~\eqref{importantDiagram}
  we are free to multiply $b_i$ by any element in
  $\Q^\times/(\Q^\times)^3$.  In this way we obtain
  \begin{equation}
    \label{b1b2}
    \begin{aligned}
      % a_1 & = \zeta_3 &
      b_1 &= 7 \theta^2 + 40 \theta + 217
      & N_{L_2/\Q}(b_1) &= 2^6 3^6 \\
      % a_2 & = 13(3 \zeta_3 + 4) &
      b_2 &= 59 \theta^2 + 314 \theta + 3011 & N_{L_2/\Q}(b_2) &= 2^3
      3^{12} 13^3
    \end{aligned}
  \end{equation}

  The connecting map $\delta_3 : E(\Q) \to H^1(\Q,E[3])$ in
  \eqref{keydiagram} may be computed as described in
  \cite[Chapter~X]{sil}.  It is given by the tangent lines at $S$ and
  $T$, i.e.  $P \mapsto (\tan_S(P),\tan_T(P))$ where
  \begin{align*}
  \tan_S (x,y) &= y - 4 \sqrt{-3} x - 52 \sqrt{-3}, \\
  \tan_T (x,y) &= y - 2 (\theta+2) x + 156 (\theta + 4)/(\theta - 4).
  \end{align*}
  The local analogue of this map is given by the same formula.

  Let $(a,b) = (a_1,b_1)$ or $(a_2,b_2)$. Since $a \in
  S^{(\phihat)}(E'/\Q)$ there exists for each prime $p$ a local point
  $P_p \in E(\Q_p)$ with $\tan_S(P_p) \equiv a
  \mod{(\Q_p(\zeta_3)^\times)^3}$. Then $(a,b)$ and
  $(\tan_S(P_p),\tan_T(P_p))$ are both local lifts of $a$.
  By~\eqref{importantDiagram} and Lemma~\ref{lem:fg} it follows that
  $b/\tan_T(P_p) \equiv \xi_p \mod{(\Q_p(\zeta_3)^\times)^3}$ for some
  $\xi_p \in \Q_p^\times/(\Q_p^\times)^3$.  By
  Definition~\ref{wpDefCT} and Lemma~\ref{prop:localpair} we have
  \begin{equation}
    \label{ctpformula}
    \langle a, a' \rangle_\CT = \sum_p \frac{1}{[\Q_p(\zeta_3):\Q_p]}
    \Ind_{\zeta_3} (\xi_p, a')_p 
\end{equation}
where $(~,~)_p$ is the Hilbert norm residue symbol on
$\Q_p(\zeta_3)$. % ^\times/(\Q_p(\zeta_3)^\times)^3$.
If $p \not= 3$ is a prime of good reduction for $E$, and $v_\pp(b)
\equiv 0 \mod{3}$ for all primes $\pp$ dividing $p$, then $p$ makes no
contribution to the sum~\eqref{ctpformula}.

Returning to our example, $E$ has minimal discriminant $-3^3 \cdot
13^3 \cdot 181$ and the norms of the $b_i$ were recorded
in~\eqref{b1b2}.  The Cassels-Tate pairing is therefore a sum of local
pairings at the primes $2,3,13$ and $181$.
Since $3$ is odd, there is no contribution from the infinite place.

\medskip

\paragraph{Contribution at $p=2$.} The local point $P = (4,2^2 + 2^6 +
O(2^8)) \in E(\Q_2)$ satisfies $\tan_S(P) \equiv a_1 \equiv a_2
\mod{(\Q_2(\zeta_3)^\times)^3}$.  Embedding $L_2$ in $\Q_2$ via
$\theta \mapsto 1+2^2+O(2^3)$ we find that $\tan_T(P) \equiv b_1
\equiv b_2 \equiv 1 \mod{(\Q_{2}^\times)^3}$. Therefore the local
pairing at $p=2$ is trivial.

\medskip

\paragraph{Contribution at $p=3$.} The local points
\begin{align*}
  P_1 & = (4,2 + 3 + 2.3^2 + O(3^5)) \in E(\Q_3) \\
  P_2 & = (3^{-2},3^{-3} + 1 + 3^2 + O(3^5)) \in E(\Q_3)
\end{align*}
satisfy $\tan_S(P_i) \equiv a_i \mod{(\Q_3(\zeta_3)^\times)^3}$ for
$i=1,2$.  Embedding $L_2$ in $\Q_3$ via $\theta \mapsto 1+2.3 + 3^3 +
O(3^4)$ we compute
\begin{align*}
  b_1/\tan_T(P_1) &\equiv 3 \mod{(\Q_{3}(\zeta_3)^\times)^3} \\
  b_2/\tan_T(P_2) &\equiv 6 \mod{(\Q_{3}(\zeta_3)^\times)^3}
\end{align*}
We recall from Section~\ref{computingLocalPairing} that
$\Q_3(\zeta_3)^\times/(\Q_3(\zeta_3)^\times)^3$ has basis
$\lambda,\eta_1,\eta_2,\eta_3$ where $\lambda = 1-\zeta_3$ and $\eta_i
= 1 - \lambda^i$. In terms of this basis we have
\begin{align*}
  3 &\equiv \lambda^2 \eta_1^2 \mod{ (\Q_3(\zeta_3)^\times)^3} &
  a_1 &\equiv \eta_1 \mod{ (\Q_3(\zeta_3)^\times)^3} \\
  6 &\equiv \lambda^2 \eta_1^2 \eta_2^2 \eta_3^2 \mod{
    (\Q_3(\zeta_3)^\times)^3} & a_2 &\equiv \eta_3^2 \mod{
    (\Q_3(\zeta_3)^\times)^3}
\end{align*}
Using~\eqref{table:pIs3Table} to compute the Hilbert norm residue
symbol, and not forgetting the factor $[\Q_3(\zeta_3):\Q_3] =2$
in~\eqref{ctpformula}, the local pairing at $p=3$ is as given
in~\eqref{localpairings-ex1}.

\medskip

\paragraph{Contribution at $p=13$.} 
We embed $L_1= \Q(\zeta_3)$ in $\Q_{13}$ via $\zeta_3 \mapsto 3 +
11.13 + O(13^2)$. The local points
\begin{align*}
P_1 & = (6,10 + 4.13 + 12.13^2 + O(13^3)) \in E(\Q_{13}) \\
P_2 & = (13,4.13 + 3.13^2 + 5.13^3 + O(13^4)) \in E(\Q_{13}) 
\end{align*}
satisfy $\tan_S(P_i) \equiv a_i \mod{(\Q_{13}^\times)^3}$ for $i=1,2$.
Embedding $L_2 = \Q(\theta)$ in $\Q_{13}$ via $\theta \mapsto 4 + 13 +
7.13^2 + O(13^3)$ we compute
\begin{align*}
b_1/\tan_T(P_1) &\equiv 1 \mod{(\Q_{13}^\times)^3} 
& a_1 &\equiv 2 \mod{(\Q_{13}^\times)^3} \\
b_2/\tan_T(P_2) &\equiv 2 \mod{(\Q_{13}^\times)^3}
& a_2 &\equiv 13^2 \mod{(\Q_{13}^\times)^3} 
\end{align*}
By Proposition~\ref{UsingCProp}(iv) we have $(2,13)_{13} = \zeta_3$.
The local pairing at $p=13$ is now given by the second matrix
in~\eqref{localpairings-ex1}.

\medskip

\paragraph{Contribution at $p=181$.} 
We embed $L_1= \Q(\zeta_3)$ in $\Q_{181}$ via $\zeta_3 \mapsto 48 +
O(181)$. We find that $a_1 \equiv a_2 \equiv 1
\mod{(\Q_{181}^\times)^3}$ and hence the local pairing at $p=181$ is
trivial.

\medskip

Adding together the local pairings at $p=3$ and $13$ gives the
(global) Cassels-Tate pairing on $S^{(\phihat)}(E'/\Q) = \langle a_1,
a_2 \rangle \subset L_1^\times / (L_1^\times)^3$.
\begin{align} \nonumber
&\text{ Local pairing at $p = 3$} &&  
\text{ Local pairing at $p = 13$} &&  
\text{ Global pairing} \\ & 
\label{localpairings-ex1}
\hspace{2em} 
\begin{array}{c|cc}
 & a_1 & a_2   \\
 \hline
 a_1  &0&1\\
 a_2  &2&1
\end{array} &&
\hspace{2em} 
\begin{array}{c|cc}
 & a_1 & a_2 \\ \hline 
 a_1  &0&0 \\
 a_2  &0&2
\end{array} &&
\hspace{1em} 
\begin{array}{c|cc}
 & a_1 & a_2 \\ \hline 
 a_1  &0&1 \\
 a_2  &2&0
\end{array}
\end{align}
Since the pairing is non-degenerate, it follows that $E(\Q)$ has rank
$0$.  Moreover the $3$-primary parts of $\Sha(E/\Q)$ and $\Sha(E'/\Q)$
are $0$ and $(\Z/3\Z)^2$. % respectively.
\end{Example}

\begin{Example}\label{SecondEx}
  Let $E$ and $E'$ be the $3$-isogenous elliptic curves labelled
  24060f1 and 24060f2 in Cremona's tables~\cite{cremdata}.
  \begin{align*}
    E: \quad y^2 &=x^3+(x+15)^2 \\
    E': \quad y^2 & =x^3-3\left(x+401/9\right)^2
  \end{align*}
  The Galois action on $E[3]$ is of type $\Z/3\Z$-nonsplit.  Indeed
  $E[3]$ is generated by
  \[ S=(0,15) \quad \text{ and } \quad T = (-90/(\theta + 2),-15
  \sqrt{-3} \theta/(\theta + 2)) \] where $\theta = \sqrt[3]{802}$.
  We set $\zeta_3 = (-1 + \sqrt{-3})/2$.  As in Section~\ref{sec:ZpZ}
  we have fields $L_1=\Q(\zeta_3)$, $L_2=\Q(\theta)$ and
  $M=\Q(\zeta_3,\theta)$. A descent by $3$-isogeny computes the Selmer
  groups
  \begin{align*}
  S^{(\phi)}(E/\Q) &= \{ 1 \} \subset (L_1^\times/(L_1^\times)^3)^{(1)} \\
  S^{(\phihat)}({E'}/\Q) &= \langle 2,3,5\rangle \subset \Q^\times/(\Q^\times)^3.
 \end{align*} 
 This gives an upper bound of $2$ for the rank of $E(\Q)$.  We seek to
 improve this bound by computing the Cassels-Tate pairing on
 $S^{(\phihat)}({E'}/\Q)$.

 We start by lifting $a_1 = 2$, $a_2 = 3$ and $a_3 = 5$ globally to
 $H^1(\Q,E[3]) \isom H$, where $H\subset K^\times/(K^\times)^3 \times
 M^\times/(M^\times)^3$ is given by Theorem~\ref{copycatTheorem}.  We
 used the existing function in Magma (again this example is too small
 for the methods of Section~\ref{sec3} to be needed) to solve the norm
 equations $N_{L_2/K}(\xi) = a_i$ for $i=1,2,3$, and then computed
 $b_i$ with $(a_i,b_i) \in H$ using Theorem~\ref{lemz3}.  Replacing
 $b_i$ by a small representative for its coset in $M^\times/
 (M^\times)^3$ we obtain
 \begin{align*}
   b_1 &= \tfrac{1}{3} (5 \zeta_3 + 5) \theta^2 + \tfrac{1}{3}
   (11 \zeta_3 - 4) \theta + \tfrac{1}{3}  (41 \zeta_3 + 290) \\
   b_2 &= \tfrac{5}{3} \zeta_3 \theta^2 - \tfrac{7}{3} \zeta_3
   \theta - \tfrac{1}{3}  (490 \zeta_3 + 213) \\
   b_3 &= \tfrac{1}{3} (7 \zeta_3 + 34) \theta^2 + \tfrac{1}{3} (66
   \zeta_3 + 317) \theta + \tfrac{1}{3} (308 \zeta_3 + 2991)
\end{align*}
It may be checked that these elements satisfy the conditions in
Theorem~\ref{copycatTheorem}.

The minimal discriminant of $E$ is $-2^4 \cdot 3^3 \cdot 5^3 \cdot
401$.  We find that $v_\pp(b_i) \equiv 0 \pmod{3}$ for all primes
${\pp}$ of $M$ not dividing 30.  The Cassels-Tate pairing is therefore
a sum of local pairings at the primes $2,3,5$ and $401$.
% Again, since $3$ is odd, we ignore the infinite place.

% The connecting map $E(\Q) \to H^1(\Q,E[3]) \isom H$ is given
% explicitly by the tangent lines at $S$ and $T$, i.e.  $P \mapsto
% (\tan_S(P),\tan_T(P))$ where
% \begin{align*}
%   \tan_S(x,y) &= y-x-15,\\
%   \tan_T (x,y) &= y + ((\theta - 1)/\sqrt{-3}) x 
%        - 15 \sqrt{-3} (\theta - 2)/(\theta + 2).
% \end{align*}

% We now compute the local pairings at $p=2,3,5,401$.

\medskip

\paragraph{Contribution at $p=2$.} 
We have $\tan_S(-S) = -30 \equiv 2 \mod{(\Q_2^\times)^3}$.  We compute
\begin{align*}
b_1/\tan_T(-S) &\equiv \zeta_3^2 \mod{(\Q_2(\zeta_3, \theta)^\times)^3} 
& a_1 &\equiv 2 \mod{(\Q_{2}^\times)^3} \\
b_2 &\equiv 1 \mod{(\Q_2(\zeta_3, \theta)^\times)^3} 
& a_2 &\equiv 1 \mod{(\Q_{2}^\times)^3} \\ 
b_3 &\equiv 1 \mod{(\Q_2(\zeta_3, \theta)^\times)^3}
& a_3 &\equiv 1 \mod{(\Q_{2}^\times)^3} 
\end{align*}
The local pairing at 2 is therefore given by the first matrix
in~\eqref{localpairings-ex2}.

\medskip

\paragraph{Contribution at $p=3$.} Let $P = (-5/2,2.3 + 2.3^2 + 3^3 +
O(3^8)) \in E(\Q_3)$. Then $\tan_S(P) \equiv 2 \equiv 5^{-1}
\mod{(\Q_3^\times)^3}$ and $\tan_S(-S) \equiv 3
\mod{(\Q_3^\times)^3}$.  We embed $L_2$ in $\Q_3$ via $\theta \mapsto
1 + 2.3 + 2.3^2 + O(3^4)$.  We recall that
$\Q_3(\zeta_3)^\times/(\Q_3(\zeta_3)^\times)^3$ has basis
$\lambda,\eta_1,\eta_2,\eta_3$ where $\lambda = 1-\zeta_3$ and $\eta_i
= 1 - \lambda^i$. We compute
\begin{align*}
  b_1/\tan_T(P) &\equiv \eta_1^2 \eta_3
  \mod{(\Q_3(\zeta_3)^\times)^3}
  & a_1 &\equiv \eta_2^2 \eta_3^2 \mod{(\Q_3(\zeta_3)^\times)^3} \\
  b_2/\tan_T(-S) &\equiv \eta_1^2
  \mod{(\Q_3(\zeta_3)^\times)^3}
  & a_2 &\equiv \lambda^2 \eta_1^2 \mod{(\Q_3(\zeta_3)^\times)^3}  \\
  b_3/\tan_T(-P) &\equiv \eta_1 \eta_3^2
  \mod{(\Q_3(\zeta_3)^\times)^3} & a_3 &\equiv \eta_2 \eta_3
  \mod{(\Q_3(\zeta_3)^\times)^3}
\end{align*}
Using~\eqref{table:pIs3Table} to compute the Hilbert norm residue
symbol, the local pairing at 3 is given by the second matrix
in~\eqref{localpairings-ex2}.

\medskip

\paragraph{Contribution at $p=5$.} 
We have $\tan_S(-S) = -30 \equiv 5 \mod{(\Q_5^\times)^3}$.  We embed
$L_2$ in $\Q_5$ via $\theta \mapsto 3 + 3.5^2 + 5^4 + O(5^5)$ and
compute
\begin{align*}
  b_1 &\equiv \zeta_3 \mod{(\Q_5(\zeta_3)^\times)^3}
  & a_1 &\equiv 1 \mod{(\Q_{5}^\times)^3} \\
  b_2 &\equiv 1 \mod{(\Q_5(\zeta_3)^\times)^3}
  & a_2 &\equiv 1 \mod{(\Q_{5}^\times)^3} \\
  b_3/\tan_T(-S) &\equiv \zeta_3
  \mod{(\Q_5(\zeta_3)^\times)^3} & a_3 &\equiv 5
  \mod{(\Q_{5}^\times)^3}
\end{align*}
The local pairing at 5 is therefore given by the third matrix
in~\eqref{localpairings-ex2}.

\medskip

\paragraph{Contribution at $p=401$.} Since $401 \equiv 2 \pmod{3}$ we
have $2,3,5 \in (\Q_{401}^\times)^3$, and so the local pairing at
$p=401$ is trivial.

\medskip

Adding together the local pairings at $p= 2$, $3$ and $5$ gives the
(global) Cassels-Tate pairing on $S^{(\phihat)}(E'/\Q) = \langle 2, 3,
5 \rangle \subset \Q^\times / (\Q^\times)^3$.
\begin{align} \nonumber  & 
%\begin{array}{c} \text{ Local pairing } & \text{ at $p = 2$ } \end{array} &&  
%\begin{array}{c} \text{ Local pairing } & \text{ at $p = 3$ } \end{array} &&  
%\begin{array}{c} \text{ Local pairing } & \text{ at $p = 5$ } \end{array} && 
\hspace{3em} p = 2 &&
\hspace{1em} p = 3 &&
\hspace{1em} p = 5 && 
% \text{ Local pairing at $p = 3$} &&  
% \text{ Local pairing at $p = 5$} &&  
\hspace{-0.5em}
\text{ Global pairing} \\ & 
\label{localpairings-ex2}
\hspace{2em} 
\begin{array}{c|ccc}
 & 2 & 3 & 5 \\  \hline
 2 &1&0&0\\
 3 &0&0&0\\
 5 &0&0&0
\end{array} &&
\begin{array}{c|cccc}
& 2 & 3 & 5 \\
 \hline
2 &2&1&1\\
3 &2&0&1\\
5 &1&2&2
\end{array} &&
\begin{array}{c|ccc}
 & 2 & 3 & 5  \\
 \hline
 2 &0&0&1\\
 3 &0&0&0\\
 5 &0&0&1
\end{array} &&
\begin{array}{c|ccc}
 & 2 & 3 & 5 \\
 \hline
 2 &0&1&2\\
 3 &2&0&1\\
 5 &1&2&0
\end{array}
\end{align}
This again shows that $\rank E(\Q)=0$, and the $3$-primary parts of
$\Sha(E/\Q)$ and $\Sha(E'/\Q)$ are $0$ and
$(\Z/3\Z)^2$. % respectively.
\end{Example}

As described in the introduction, Eroshkin found five examples of
elliptic curves $E/\Q$ with torsion subgroup $\Z/3\Z$ and rank at
least $13$. We now consider the first of these examples.  The other
examples are similar, and are treated in detail in \cite[Section
6.1]{phdMonique}.

\begin{Example}
\label{ThirdEx}
Let $E/\Q$ be the elliptic curve $y^2 + A_1 xy + A_3 y = x^3$ where
$A_1 = 10154960719$ and $A_3 = -66798078951809458114391930400$.  The
primes of bad reduction for $E$ are those appearing in the following
prime factorisations.  \small
  \begin{align*} A_3 = -2^5 \cdot 3^3 \cdot 5^2 \cdot 7^2 \cdot 11
    \cdot 13 \cdot 17 \cdot 19 \cdot 23 \cdot 29 \cdot 31 \cdot 37
    \cdot 41 \cdot 4&3
    \cdot 47 \cdot 53 \cdot 59 \cdot 61 \cdot 113, \\
    A_1^3 - 27 A_3 = 197 \cdot 317 \cdot 3313949 \cdot 2831657657
    \cdot 4&864617187.
  \end{align*}
  \normalsize

% Let $P_1, \ldots, P_{13}$ be the 
%   The 
% known independent points of infinite order in $E(\Q)$,
% as listed at \cite{dujweb}.
%  are:
%   \begin{align*}
%     P_{1} &= ( 22385162997659061600 , 52625844692710748830154784000 ), \\
%     P_{2} &= ( 34121356867937069640 , 103633234297022468496697260840 ), \\
%     P_{3} &= ( 8923442541064091040 , 17286189547043616229977186240 ), \\
%     P_{4} &= ( 5813934277239629100 , 18424262870234651453351233500 ), \\
%     P_{5} &= ( 18430587542804127600 , 39228780433391713968199578000 ), \\
%     P_{6} &= ( 6629770994718537120 , 16809101795665860615489347520 ), \\
%     P_{7} &= ( -28219028077355824800 , 270194450725400678854982841600 ), \\
%     P_{8} &= ( 10205673044535503990 , 19026943580934583948460342900 ), \\
%     P_{9} &= ( -26941722141580941240 , 267203106074722424190858106560 ), \\
%     P_{10} &= ( 7410830460204436650 , 16383622996035063617949015300 ), \\
%     P_{11} &= ( -35828808535202217150 , 234522026020484344865490594750 ), \\
%     P_{12} &= ( 4732967352735606600 , 23287727366639341635234744000 ), \\
%     P_{13} &= ( 18475849916970533100 , 39370225909289341263266862000 ).
% \end{align*} 

  The Galois action on $E[3]$ is of type $\Z/3\Z$-nonsplit.  Indeed
  $E[3]$ is generated by $S=(0,0)$ and $T = (3 A_3/(\theta - A_1),A_3
  (\zeta_3 \theta - A_1)/(\theta - A_1))$ where $\theta =
  \sqrt[3]{A_1^3 - 27 A_3}$. Let $\phi : E \to E'$ be the $3$-isogeny
  with kernel generated by $S$. A descent by $3$-isogeny
  \cite[Proposition 1.2]{ctpps} shows that
  \[ S^{(\phihat)}({E'}/\Q) = \left\{ x \in \Q^\times/(\Q^\times)^3
    \Bigg| \begin{array}{cl} v_p(x) \equiv 0 \!\!\!
      \pmod{3} & \text{ for all } p \nmid A_3 \\
      x \in (\Q_p^\times)^3 & \text{ for all } p \mid (A_1^3 - 27 A_3)
    \end{array} \right\}. \]
  Noting that only one % the largest
  of the prime factors of $A_1^3 - 27 A_3$ is
  congruent to $1$ mod~$3$, we find that $S^{(\phihat)}({E'}/\Q)$
  % there is only one linear condition. Thus $S^{(\phihat)}({E'}/\Q)$
  is the $18$-dimensional $\F_3$-vector space with basis
\begin{equation}
\label{selbasis}
\begin{aligned}
  2, \,\, 5, \,\, 11, \,\, 17, \,\, & \, 31, \,\, 47, \,\, 53, \,\,
  3^2 \cdot 7, \,\, 3 \cdot 13, \,\, 3 \cdot 19, \,\, 3 \cdot 23, \,\,
  \\ & 3 \cdot 29, \,\, 3 \cdot 37, \,\, 3 \cdot 41, \,\, 3 \cdot 43,
  \,\, 3 \cdot 59, \,\, 3^2 \cdot 61, \,\, 3 \cdot 113.
\end{aligned}
\end{equation}
By the analogue (for $n=3$) of \cite[Theorem 1]{5descent}, or by
Cassels' formula~\cite{bsd}, it follows that $S^{(\phi)}(E/\Q)$ is
trivial. This gives an upper bound of $17$ for the rank of $E(\Q)$.
We improve this bound by computing the Cassels-Tate pairing on the
subspace of $S^{(\phihat)}({E'}/\Q)$ generated by the first $5$ basis
elements in~\eqref{selbasis}, say $a_1, \ldots, a_5$.

As in Section~\ref{sec:ZpZ} we have fields $L_1=\Q(\zeta_3)$,
$L_2=\Q(\theta)$ and $M=\Q(\zeta_3,\theta)$.  In Example~\ref{NormEx2}
we solved one of the norm equations $N_{L_2/\Q}(\xi)=a_i$. The other
cases are similar.
% We used the method of Section~\ref{sec3} to solve the norm equations
% $N_{L_2/\Q}(\xi) = a_i$. For example, the equation
% $N_{L_2/\Q}(\xi)=17$ was solved in Example~\ref{NormEx2}.
We then used Theorem~\ref{lemz3} to compute $b_i \in M$ with
$(a_i,b_i) \in H$, where $H$ is as defined in
Theorem~\ref{copycatTheorem}.  So that they could sensibly be recorded
in the paper, we went to some effort to simplify the $b_i$, both by
multiplying by elements of $(L_1^\times/(L_1^\times)^3)^{(1)}$ and by
finding small representatives modulo cubes.  \small
 \begin{align*}
  b_1 &=   80506656009 \theta^2 - 1176048899716052084841 \theta 
    \\ & - 14935178208744640295856847246416 \zeta_3 - 
            15036024242599209733354645439703, \\
  b_2 &=   14726363049 \theta^2 - 79874874765966026529 \theta 
    \\ & + 8657187467761497385350294134040 \zeta_3 - 
            8434480171840925245748610923511, \\
  b_3 &=   218823372684 \theta^2 - 4630953487853681932716 \theta 
    \\ & + 34676125489353056066296086569091 \zeta_3 + 
        60807466313987014328526766460838, \\
  b_4 &=   286372386666 \theta^2 - 1448511948608043607524 \theta
    \\ & - 57528276376283017594756117712901 \zeta_3 - 
        38980928584242432627609103951923, \\
  b_5 &=   332611290882 \theta^2 + 1168159925437207764516 \theta 
    \\ & - 67751649380200776098612752578639 \zeta_3 + 
        71449768157279254278949836738165
\end{align*}
\normalsize 
We find that $v_\pp(b_i) \equiv 0 \pmod{3}$ for all primes $\pp$ of
$M$ not dividing $a_i$. Therefore only the bad primes for $E$
contribute to the Cassels-Tate pairing.  The local conditions used to
compute $S^{(\phihat)}(E'/\Q)$ show that its elements are locally
trivial at the primes dividing $A_1^3 - 27 A_3$. So we only need to
compute the local pairings at the primes dividing $A_3$.

Let $P_1, \ldots, P_{13}$ be the known independent points of infinite
order in $E(\Q)$, as listed on Dujella's website~\cite{dujweb}.

For the primes $p$ with $p \equiv 1 \mod{3}$ we find that
$\tan_S(P_j)$ generates $\Z_p^\times/(\Z_p^\times)^3$ where $j =
6,4,1,1,2,2,7$ for $p = 7,13,19,31,37,43,61$. Moreover $\tan_T(P_j)$
is a unit mod cubes at the primes dividing $p$.  So we only need to
consider the primes that additionally divide one of the $a_i$. The
only such prime is $31$. Embedding $M$ in $\Q_{31}$ via $\zeta_3
\mapsto 5 + 14.31 + O(31^2)$ and $\theta \mapsto 1 - 2.31^2 +
O(31^4)$, we compute
\begin{align*}
  b_1 &\equiv 5^2 \mod{(\Q_{31}^\times)^3}
  & a_1 &\equiv 1 \mod{(\Q_{31}^\times)^3} \\
  b_2/\tan_T(-P_1) &\equiv 1 \mod{(\Q_{31}^\times)^3}
  & a_2 &\equiv 5 \mod{(\Q_{31}^\times)^3} \\
  b_3/\tan_T(-P_1) &\equiv 1 \mod{(\Q_{31}^\times)^3}
  & a_3 &\equiv 5 \mod{(\Q_{31}^\times)^3} \\
  b_4/\tan_T(P_1) &\equiv 1 \mod{(\Q_{31}^\times)^3}
  & a_4 &\equiv 5^2 \mod{(\Q_{31}^\times)^3} \\
  b_5/\tan_T(P_6) &\equiv 31^2 \mod{(\Q_{31}^\times)^3} & a_5 &\equiv
  31 \mod{(\Q_{31}^\times)^3}
\end{align*}
This gives the local pairing at $p=31$ as recorded below.

For the primes $p$ with $p \equiv 2 \mod{3}$ the group
$\Z_p^\times/(\Z_p^\times)^3$ is trivial. So we only need to consider
those primes $p$ that additionally divide one of $a_1, \ldots, a_5$.
We find that $\tan_S(P_j) \equiv p \mod (\Q_p^\times)^3$ where
$j=5,2,1,5$ for $p=2,5,11,17$. The unique embedding of $L_2$ in $\Q_p$
determines an embedding of $M$ in $\Q_p(\zeta_3)$.  Then
$b_i/\tan_T(P_j)^{v_p(a_i)}$ takes the following values mod
$(\Q_p(\zeta_3)^\times)^3$.
\[ \begin{array}{c|cccc} 
     &  p = 2     & p = 5 & p = 11 & p = 17 \\ \hline
i= 1 &  1         &   1   &\zeta_3^2& (\zeta_3 + 3)^2 \\
i= 2 &  1         &   1   &   1   & 1 \\
i= 3 & \zeta_3  &   1   &   1   &  (\zeta_3 + 3)^2 \\
i= 4 & \zeta_3^2    &   1   &\zeta_3^2 & 1 \\
i= 5 & \zeta_3^2    &\zeta_3^2&\zeta_3&  \zeta_3 + 3 
\end{array} \]
This gives the local pairings at $p=2,5,11,17$ as recorded below.

Finally, when $p=3$, we find that $\tan_S(P_8) \equiv 2 \mod
(\Q_3^\times)^3$, whereas the elements $b_1, \ldots, b_5$ and
$\tan_T(P_8)$ all belong to the subgroup of
$\Q_3(\zeta_3)^\times/(\Q_3(\zeta_3)^\times)^3$ generated by $\eta_3 =
1 - (1-\zeta_3)^3$. The local pairing at $p=3$ is therefore trivial.

Adding together the local pairings gives the (global) Cassels-Tate
pairing on the $5$-dimensional subspace $\langle 2,5,11,17,31 \rangle$
of $S^{(\phihat)}(E'/\Q) \subset \Q^\times / (\Q^\times)^3$.  The fact
we obtain an alternating matrix provides some check on our
calculations.
\begin{equation*}
\begin{array}{c@{\quad}c@{\quad}c}
  \hspace{-2em} \text{ Local pairing at $p = 2$} &  
  \text{ Local pairing at $p = 5$} &  
  \text{ Local pairing at $p = 11$} \hspace{-2em} \\
  \begin{array}{c|ccccc}
    & 2 & 5 & 11 & 17 & 31 \\  \hline
    2 &0&0&0&0&0\\
    5 &0&0&0&0&0\\
   11 &2&0&0&0&0\\
   17 &1&0&0&0&0\\
   31 &1&0&0&0&0
  \end{array} &
  \begin{array}{c|cccccc}
    & 2 & 5 & 11 & 17 & 31 \\  \hline
    2 &0&0&0&0&0\\
    5 &0&0&0&0&0\\
   11 &0&0&0&0&0\\
   17 &0&0&0&0&0\\
   31 &0&2&0&0&0
  \end{array} &
  \begin{array}{c|ccccc}
    & 2 & 5 & 11 & 17 & 31 \\  \hline
    2 &0&0&1&0&0\\
    5 &0&0&0&0&0\\
   11 &0&0&0&0&0\\
   17 &0&0&1&0&0\\
   31 &0&0&2&0&0
 \end{array}
 \end{array}
\end{equation*}
\begin{equation*}
\begin{array}{c@{\quad}c@{\quad}c}
\hspace{-2em} \text{ Local pairing at $p = 17$} &  
\text{ Local pairing at $p = 31$} &  
\text{ Global pairing } \hspace{-2em} \\
\begin{array}{c|ccccc}
  & 2 & 5 & 11 & 17 & 31 \\  \hline
  2 &0&0&0&2&0\\
  5 &0&0&0&0&0\\
 11 &0&0&0&2&0\\
 17 &0&0&0&0&0\\
 31 &0&0&0&1&0
\end{array} &
\begin{array}{c|cccccc}
  & 2 & 5 & 11 & 17 & 31 \\  \hline
  2 &0&0&0&0&2\\
  5 &0&0&0&0&0\\
 11 &0&0&0&0&0\\
 17 &0&0&0&0&0\\
 31 &0&1&1&2&0
\end{array} &
\begin{array}{c|ccccc}
  & 2 & 5 & 11 & 17 & 31 \\  \hline
  2 &0&0&1&2&2\\
  5 &0&0&0&0&0\\
 11 &2&0&0&2&0\\
 17 &1&0&1&0&0\\
 31 &1&0&0&0&0
 \end{array}
 \end{array}
\end{equation*}
Since the Cassels-Tate pairing on this $5$-dimensional subspace of
$S^{(\phihat)}(E'/\Q)$ has rank $4$, it follows that $\rank E(\Q) =
13$.  Moreover the $3$-primary parts of $\Sha(E/\Q)$ and $\Sha(E'/\Q)$
are $0$ and $(\Z/3\Z)^4$.  The $18 \times 18$ matrix (still of rank~4)
giving the Cassels-Tate pairing on all of $S^{(\phihat)}(E'/\Q)$ is
recorded in \cite[Example 6.1.2]{phdMonique}.
\end{Example}

\bibliographystyle{alpha}
\bibliography{references}{}

\end{document}